\documentclass{article}
\usepackage{amsmath,amscd,amsfonts,amssymb,amsthm}
\usepackage{graphicx}
\usepackage{epstopdf}
\usepackage{a4wide}
\usepackage{graphicx}
\usepackage{subfigure}
\usepackage[latin1]{inputenc}
\usepackage{color}

\newcommand{\subj}[1]{\par\noindent{\bf Mathematics Subject Classification 2010: }#1.}
\newcommand{\keyw}[1]{\par\noindent{\bf Keywords: }#1.}

\newcommand{\fonction}[5]{\begin{array}[t]{lrcl}#1 :&#2 &\longrightarrow &#3\\&#4& \longmapsto &#5 \end{array}}

\newcommand{\R}{\mathbb{R}}
\newcommand{\C}{\mathbb{C}}
\newcommand{\N}{\mathbb{N}}

\theoremstyle{definition}
\newtheorem{definition}{Definition}
\newtheorem{theorem}{Theorem}
\newtheorem{example}{Example}
\newtheorem{lemma}{Lemma}
\theoremstyle{remark}
\newtheorem{remark}{Remark}

\def\a{\alpha}
\def\r{\rho}
\def\t{\tau}
\def\DS{\displaystyle}
\def\LD{{^CD_{a+}^{\a,\r}}}
\def\RD{{^CD_{b-}^{\a,\r}}}
\def\LI{{I_{a+}^{\a,\r}}}
\def\RI{{I_{b-}^{\a,\r}}}

\begin{document}

\title{Fractional differential equations with dependence on the Caputo--Katugampola derivative}

\author{Ricardo Almeida$^1$\\
{\tt ricardo.almeida@ua.pt}
\and
Agnieszka B. Malinowska$^{2}$\\
{\tt a.malinowska@pb.edu.pl}
\and
Tatiana Odzijewicz$^{3}$\\
{\tt tatiana.odzijewicz@sgh.waw.pl}}

\date{$^1$Center for Research and Development in Mathematics and Applications (CIDMA)\\
Department of Mathematics, University of Aveiro, 3810--193 Aveiro, Portugal\\
$^2$Faculty of Computer Science, Bia\l ystok University of Technology,\\ 15-351 Bia\l ystok, Poland\\
$^3$Department of Mathematics and Mathematical Economics,\\
Warsaw School of Economics\\
 02-554 Warsaw, Poland}

\maketitle


\begin{abstract}

In this paper we present a new type of fractional operator, the Caputo--Katugampola derivative. The Caputo and the Caputo--Hadamard fractional derivatives are special cases of this new operator. An existence and uniqueness theorem for a fractional Cauchy type problem, with dependence on the Caputo--Katugampola derivative, is proven. A decomposition formula for the Caputo--Katugampola derivative is obtained. This formula allows us to provide a simple numerical procedure to solve the fractional differential equation.

\end{abstract}

\subj{26A33, 34A08, 34K28}

\keyw{fractional calculus, fractional differential equations, Caputo--Katugampola derivative, numerical methods}


\section{Introduction}

Fractional calculus is a natural extension of ordinary calculus, where integrals and derivatives are defined for arbitrary real orders.  Since $17th$ century, when fractional calculus was born, several different derivatives have been introduced: Riemann-Liouville, Hadamard, Grunwald--Letnikov, Caputo, just to mention a few \cite{Kilbas,Podlubny,Samko}, each of them with its own advantages and disadvantages. The choice of an appropriate fractional derivative (or integral) depends on the considered system, and for this reason we find a large number of works devoted to different fractional operators. On the other hand, it is natural to look for and study generalized fractional operators, for which the known ones are particular cases. Recently, U. Katugampola presented new types of fractional operators, which generalize both the Riemann--Liouville and Hadamard fractional operators \cite{Katugampola1,Katugampola2,Katugampola3}.

In this paper we present a new fractional operator, called the Caputo--Katugampola derivative, which generalizes the concept of Caputo and Caputo--Hadamard fractional derivatives (Section~\ref{sec:FC}). The new operator is the left inverse of the Katugampola fractional integral and keeps some of the fundamental properties of the Caputo and Caputo--Hadamard fractional derivatives. In Section~\ref{sec:FDE} we study a fractional differential equation (FDE), where the differential operator is the Caputo--Katugampola derivative, showing an existence and uniqueness theorem for the Cauchy type problem. Section~\ref{sec:approx} provides an approximated expansion of the Caputo--Katugampola derivative, that involves only the first-order derivative. Applying this formula to a given Caputo--Katugampola fractional Cauchy problem we obtain a sequence of systems of $N+1$ ordinary differential equations with $N+1$ initial conditions, where $N\in \N$. A sequence $(x_N)$ of the solutions to those systems of ODE's converges to the solution of the fractional problem. The numerical procedure to solve the fractional differential equation is presented in Section~\ref{Sec:examples} throughout an example. Finally, Section~\ref{con} contains some conclusions.

\section{Caputo--Katugampola fractional derivative}\label{sec:FC}

In \cite{Katugampola1}, the author introduced new fractional integral operators, called the Katugampola fractional integrals, in the following way:
$${I_{a+}^{\a,\r}}x(t)=\frac{\r^{1-\a}}{\Gamma(\alpha)}\int_a^t \frac{\t^{\r-1}}{(t^\r-\t^\r)^{1-\a}}x(\t) d\tau,$$
$${I_{b-}^{\a,\r}}x(r)=\frac{\r^{1-\a}}{\Gamma(\alpha)}\int_t^b \frac{\t^{\r-1}}{(\t^\r-t^\r)^{1-\a}}x(\t) d\tau,$$
where $0<a<b<\infty$, $x:[a,b]\rightarrow\mathbb{R}$ is an integrable function, and $\alpha\in(0,1)$ and $\rho>0$ two fixed real numbers.
Then, in \cite{Katugampola2}, it appeared a generalization to { the} Riemann--Liouville and Hadamard fractional derivatives, called the Katugampola fractional derivatives:
$${D_{a+}^{\a,\r}}x(t)=\frac{\r^\a}{\Gamma(1-\alpha)}t^{1-\r}\frac{d}{dt}\int_a^t \frac{\t^{\r-1}}{(t^\r-\t^\r)^\a}x(\t) d\tau,$$
and
$${D_{b-}^{\a,\r}}x(t)=\frac{-\r^\a}{\Gamma(1-\alpha)}t^{1-\r}\frac{d}{dt}\int_t^b \frac{\t^{\r-1}}{(\t^\r-t^\r)^\a}x(\t) d\tau.$$
The relation between these two fractional operators is the following:
$${D_{a+}^{\a,\r}}x(t)=t^{1-\r}\frac{d}{dt}{I_{a+}^{1-\a,\r}}x(t) \quad \mbox{and} \quad {D_{b-}^{\a,\r}}x(r)=-t^{1-\r}\frac{d}{dt}{I_{b-}^{1-\a,\r}}x(r).$$
Also, when $\a>0$ is an integer, these fractional integrals are generalizations of the n-fold integrals
$$\int_a^t \t_1^{\r-1}\, d\t_1\,\int_a^{\t_1} \t_2^{\r-1}\, d\t_2\,\ldots \int_a^{\t_{n-1}} \t_n^{\r-1}x(\t_n)\, d\t_n$$
and
$$\int_t^b \t_1^{\r-1}\, d\t_1\,\int_{\t_1}^b \t_2^{\r-1}\, d\t_2\,\ldots \int_{\t_{n-1}}^b \t_n^{\r-1}x(\t_n)\, d\t_n,$$
respectively.

Motivated by those definitions, we introduce a new fractional derivative as follows.

\begin{definition}\label{CKFD}
Let $0<a<b<\infty$, $x:[a,b]\rightarrow\mathbb{R}$ be an integrable function, and $\alpha\in(0,1)$ and $\rho>0$ two fixed reals. The left and right Caputo--Katugampola fractional derivatives of order $\a$ are defined by
$$\LD x(t) ={D_{a+}^{\a,\r}}[x(t)-x(a)]=\frac{\r^\a}{\Gamma(1-\alpha)}t^{1-\r}\frac{d}{dt}\int_a^t \frac{\t^{\r-1}}{(t^\r-\t^\r)^\a}[x(\t)-x(a)] d\tau,$$
and
$$\RD x(t)={D_{b-}^{\a,\r}}[x(t)-x(b)]=\frac{-\r^\a}{\Gamma(1-\alpha)}t^{1-\r}\frac{d}{dt}\int_t^b \frac{\t^{\r-1}}{(\t^\r-t^\r)^\a}[x(\t)-x(b)] d\tau,$$
respectively.
\end{definition}

Observe that, when $x(a)=0$ (resp. $x(b)=0$), the left (resp. right) Caputo--Katugampola and the Katugampola fractional derivatives coincide.
Moreover, there exists a relation between these { two types} of fractional derivatives. Namely, if both types of fractional derivative exist, then
\begin{equation}\label{eq:derivat}
\LD x(t)={D_{a+}^{\a,\r}}[x(t)-x(a)]={D_{a+}^{\a,\r}}x(t)-{D_{a+}^{\a,\r}}x(a)={D_{a+}^{\a,\r}}x(t)-\frac{x(a)\r^\a}{\Gamma(1-\a)}(t^\r-a^\r)^{-\a},
\end{equation}
and
$$\RD x(t)={D_{b-}^{\a,\r}}[x(t)-x(b)]={D_{b-}^{\a,\r}}x(t)-{D_{b-}^{\a,\r}}x(b)={D_{b-}^{\a,\r}}x(t)-\frac{x(b)\r^\a}{\Gamma(1-\a)}(b^\r-t^\r)^{-\a}.$$

\begin{theorem}\label{teo:equiv} Let $x\in C^1([a,b]).$ Then,
$$\LD x(t) =\frac{\r^\a}{\Gamma(1-\alpha)}\int_a^t (t^\r-\t^\r)^{-\a}x'(\t) d\tau,$$
and
$$\RD x(t)=\frac{-\r^\a}{\Gamma(1-\alpha)}\int_t^b (\t^\r-t^\r)^{-\a}x'(\t) d\tau.$$
\end{theorem}

\begin{proof}
Taking
$$u'(\t)=\t^{\r-1}(t^\r-\t^\r)^{-\a} \quad\mbox{and} \quad v(\t)=x(\t)-x(a),$$
and then integrating by parts we obtain
$$\LD x(t) =\frac{\r^\a}{\Gamma(1-\alpha)}t^{1-\r}\frac{d}{dt}\int_a^t u'(\tau)v(\tau)d\tau.$$

We may now differentiate the integral, with respect to variable $t$, to conclude that
$$\LD x(t) =\frac{\r^\a}{\Gamma(1-\alpha)}\int_a^t (t^\r-\t^\r)^{-\a}x'(\t) d\tau.$$
The second formula can be proven in a similar way.
\end{proof}

\begin{remark}
Observe that, from Theorem~\ref{teo:equiv}, we immediately obtain the following:
\begin{equation}\label{Derivative_Integral}\LD x(t) ={I_{a+}^{1-\a,\r}}\left(\t^{1-\r}\frac{d}{d\t}x\right)(t) \quad \mbox{and} \quad \RD x(t) ={I_{b-}^{1-\a,\r}}\left(-\t^{1-\r}\frac{d}{d\t}x\right)(t),\end{equation}
where ${I_{a+}^{1-\a,\r}}$ and ${I_{b-}^{1-\a,\r}}$ denote the left and right Katugampola fractional integrals.
\end{remark}

\begin{remark} Although in Definition~\ref{CKFD} we put $\a\in(0,1)$, it is possible to define the Caputo--Katugampola fractional derivatives for an arbitrary $\a>0$.  Let $n$ be the smallest integer greater than $\a$. Then, the left and right Caputo--Katugampola fractional derivatives of order $\a>0$ are defined by
$$\LD x(t) =\frac{\r^{\a-n+1}}{\Gamma(n-\alpha)}\int_a^t \frac{\t^{(\r-1)(1-n)}}{(t^\r-\t^\r)^{\a-n+1}}x^{(n)}(\t)d\tau,$$
and
$$\RD x(t)=\frac{(-1)^n\r^{\a-n+1}}{\Gamma(n-\alpha)}\int_t^b \frac{\t^{(\r-1)(1-n)}}{(\t^\r-t^\r)^{\a-n+1}}x^{(n)}(\t)d\tau,$$
respectively. In this case, when $\r=1$, we obtain the left and right Caputo fractional derivatives (cf. \cite{Kilbas}), and for $\r\to0^+$, attending that
$\lim_{\r\to 0^+} (t^\r-\t^\r)/\r=\ln(t/\t)$,
we obtain the left and right  Caputo--Hadamard fractional derivatives as defined in \cite{Baleanu1,Baleanu2}.
\end{remark}

Now let us prove that the left and the right Katugampola fractional integrals are bounded in the space of continuous functions.
\begin{theorem}\label{lem:contInt}
The operators $\LI$ and $\RI$ are linear and bounded from $C([a,b])$ to $C([a,b])$, i.e.,
\begin{equation}\label{eq:boundedL}
\left\|\LI x\right\|_C\leq K_{\alpha,\rho}\left\|x\right\|_C,~\left\|\RI x\right\|_C\leq K_{\alpha,\rho}\left\|x\right\|_C,
~\textnormal{with } K_{\alpha,\rho}=\frac{\rho^{-\alpha}}{\Gamma(\alpha+1)}(b^{\rho}-a^{\rho})^{\alpha},
\end{equation}
where $\left\|x\right\|_C=\max\limits_{t\in [a,b]} \left|x(t)\right|$.
\end{theorem}
\begin{proof}
For any $x\in C([a,b])$, one has
$$\begin{array}{ll}
\displaystyle\left|\frac{\rho^{1-\alpha}}{\Gamma(\alpha)}\int_a^t\frac{\tau^{\rho-1}}{(t^{\rho}-\tau^{\rho})^{1-\alpha}}x(\tau)\; d\tau\right|&
\displaystyle\leq\frac{\rho^{1-\alpha}}{\Gamma(\alpha)}\left\|x\right\|_C\int_a^t\frac{\tau^{\rho-1}}{(t^{\rho}-\tau^{\rho})^{1-\alpha}}\; d\tau\\
&\displaystyle\leq\frac{\rho^{-\alpha}}{\Gamma(\alpha+1)}(b^{\rho}-a^{\rho})^{\alpha}\left\|x\right\|_C.
\end{array}$$
In the case of the right Katugampola fractional integral, the proof is similar.
\end{proof}
Using the above lemma, one can show the following theorem concerning continuity of the Caputo--Katugampola fractional derivatives.

\begin{theorem}\label{teo:bounds} Let $x\in C^1([a,b]).$ Then, the Caputo--Katugampola fractional derivatives $\LD x$ and $\RD x$ exist and are continuous on the interval $[a,b]$. Moreover,  $\LD x(t)=0$ at $t=a$, and $\RD x(t)=0$ at $t=b$.
\end{theorem}

\begin{proof} Knowing that
$$\LD x(t) ={I_{a+}^{1-\a,\r}}\left(\t^{1-\r}\frac{d}{d\t}x\right)(t),$$
by Theorem~\ref{lem:contInt}, we get a continuity of $\LD x$. In order to prove the second part, we start with the formula
$$\LD x(t) =\frac{\r^\a}{\Gamma(1-\alpha)}\int_a^t (t^\r-\t^\r)^{-\a}x'(\t) d\tau,$$
that holds for all $t\in [a,b]$. Now, letting $M=\max_{t\in[a,b]}|t^{1-\r}x'(t)|$, we obtain
$$\left|\LD x(t)\right|\leq M\frac{\r^{\a-1}}{\Gamma(2-\a)}(t^\r-a^\r)^{1-\a}.$$ Therefore, $\LD x(t)=0$ at $t=a$.
One can obtain the results for the right Caputo-Katugampola fractional derivative in a similar way.
\end{proof}

In the space $C^1([a,b])$ we consider the norm
$$\|x\|_{C^1}:=\max_{t\in[a,b]}|x(t)|+\max_{t\in[a,b]}|x'(t)|.$$
Also, we define the two subspaces $C_a([a,b])$  and $C_b([a,b])$  of the space $C([a,b])$ by
$$C_a([a,b]):=\left\{x\in C([a,b])\, : \, x(a)=0\right\}, $$
and
$$C_b([a,b]):=\left\{x\in C([a,b])\, : \, x(b)=0\right\},$$
endowed with the norm
$$\|x\|_{C}:=\max_{t\in[a,b]}|x(t)|.$$

\begin{theorem} The Caputo--Katugampola fractional derivatives $\LD$ and $\RD$ are bounded operators from $C^1([a,b])$ to $C_a([a,b])$ and $C_b([a,b])$, respectively. Moreover,
$$\left\|\LD x\right\|_{C}\leq M \|x\|_{C^1} \quad \mbox{and} \quad \left\|\RD x\right\|_{C}\leq M \|x\|_{C^1},$$
where
$$M= \frac{\r^{\a-1}(b^\r-a^\r)^{1-\a}}{\Gamma(2-\a)}\max\{a^{1-\r},b^{1-\r}\}.$$
\end{theorem}

\begin{proof} It is an immediate consequence of \eqref{Derivative_Integral} and Theorem~\ref{lem:contInt}.
\end{proof}

Under appropriate assumptions, the Caputo--Katugampola fractional derivatives and Katugampola fractional integrals are inverse operations of each other. In fact, we have the following.

\begin{theorem}\label{thm:DerInt} Let $x\in C([a,b]).$ Then,
$$\LD {I_{a+}^{\a,\r}} x(t)=x(t).$$
\end{theorem}

\begin{proof} Using formula \eqref{eq:derivat}, by Theorem 3.2 in \cite{Katugampola2}, we have
$$\begin{array}{ll}
\DS \LD {I_{a+}^{\a,\r}} x(t)&=\DS {D_{a+}^{\a,\r}} {I_{a+}^{\a,\r}} x(t)-\frac{\r^\a}{\Gamma(1-\a)}(t^\r-a^\r)^{-\a}\left[{I_{a+}^{\a,\r}} x(t)\right]|_{t=a}\\
&=\DS  x(t)-\frac{\r^\a}{\Gamma(1-\a)}(t^\r-a^\r)^{-\a}\left[{I_{a+}^{\a,\r}} x(t)\right]|_{t=a}.\\
\end{array}$$
On the other hand, since $x$ is a continuous function, there exists $M=\max_{t\in[a,b]}|x(t)|$. Then,
$$\begin{array}{ll}
\DS \left|{I_{a+}^{\a,\r}} x(t)\right|&\DS\leq M\frac{\r^{1-\a}}{\Gamma(\alpha)}\int_a^t \frac{\t^{\r-1}}{(t^\r-\t^\r)^{1-\a}} d\tau\\
&\DS =M\frac{\r^{-\a}}{\Gamma(1+\alpha)}(t^\r-a^\r)^{\a},\\
\end{array}$$
and ${I_{a+}^{\a,\r}} x(t)=0$ at $t=a$. Thus, $\LD {I_{a+}^{\a,\r}} x(t)=x(t)$.
\end{proof}

\begin{theorem}\label{teo:ID} Let $x\in C^1([a,b]).$ Then,
$${I_{a+}^{\a,\r}}\, \LD x(t)=x(t)-x(a).$$
\end{theorem}

\begin{proof} By Theorem 4.1 in \cite{Katugampola1} and by the definition of the Katugampola fractional integral, we have
$${I_{a+}^{\a,\r}}\, \LD x(t)={I_{a+}^{\a,\r}}{I_{a+}^{1-\a,\r}}(t^{1-\r}x'(t))={I_{a+}^{1,\r}}(t^{1-\r}x'(t))$$
$$=\int_a^t \t^{\r-1}\t^{1-\r}x'(\t)d\t=x(t)-x(a).$$
\end{proof}

\begin{lemma}\label{lemma:power} For $v>0$, define
$$x(t)=\left(\frac{t^\r-a^\r}{\r}\right)^v \quad \mbox{and} \quad y(t)=\left(\frac{b^\r-t^\r}{\r}\right)^v.$$
Then,
$$\LD x(t)=\frac{\r^{\a-v}\Gamma(1+v)}{\Gamma(1-\a+v)}(t^\r-a^\r)^{v-\a}  \quad \mbox{and} \quad
\RD y(t)=\frac{\r^{\a-v}\Gamma(1+v)}{\Gamma(1-\a+v)}(b^\r-t^\r)^{v-\a}.$$
\end{lemma}

\begin{proof} By Theorem \ref{teo:equiv},
$$\LD x(t) =\frac{\r^\a}{\Gamma(1-\alpha)}\int_a^t (t^\r-\t^\r)^{-\a}v\t^{\r-1}\left(\frac{\t^\r-a^\r}{\r}\right)^{v-1} d\tau.$$
Performing the change of variables $u=(\t^\r-a^\r)/(t^\r-a^\r)$, and using the Beta function $B(\cdot,\cdot)$, we get
$$\begin{array}{ll}
\DS\LD x(t) &=\DS \frac{\r^{\a-v}v}{\Gamma(1-\alpha)}(t^\r-a^\r)^{v-\a}\int_0^1 (1-u)^{-\a}u^{v-1} du\\
&=\DS \frac{\r^{\a-v}v}{\Gamma(1-\alpha)}(t^\r-a^\r)^{v-\a}B(1-\a,v)\\
&=\DS \frac{\r^{\a-v}v}{\Gamma(1-\alpha)}(t^\r-a^\r)^{v-\a}\frac{\Gamma(1-\a)\Gamma(v)}{\Gamma(1-\a+v)}\\
&=\DS \frac{\r^{\a-v}\Gamma(1+v)}{\Gamma(1-\a+v)}(t^\r-a^\r)^{v-\a}.
\end{array}$$
The second formula can be proven in a similar way.
\end{proof}

\begin{remark} If $\r=1$, then we obtain the power function $x(t)=(t-a)^v$, and as $\r\to0^+$, we  get $x(t)=(\ln(t/a))^v$. The expressions obtained in Lemma \ref{lemma:power} coincide with the ones given in Property~2.16 \cite{Kilbas} and Property~2.6 \cite{Baleanu2}.
\end{remark}

Let us recall the definition of the Mittag--Leffler function.

\begin{definition} The Mittag--Leffler function with dependence on two parameters $\mu$ and $\beta$ is defined by the series
$$E_{\mu,\beta}(x)=\sum_{k=0}^\infty\frac{x^k}{\Gamma(\mu k+\beta)},\quad \mu>0 \, \mbox{and} \, \beta>0.$$
where $\mu, \beta \in \C$, $Re(\mu)>0$, $Re(\beta)>0$. When $\mu, \beta \in \R$ and $\mu, \beta>0$, the series is convergent.
As a particular case, when $\beta=1$, we obtain
$$E_{\mu}(x)=\sum_{k=0}^\infty\frac{x^k}{\Gamma(\mu k+1)},\quad \mu>0.$$
\end{definition}

When $\mu=1$, the Mittag--Leffler function reduces to the exponential function, $E_{1}(x)=\exp(x)$.

\begin{lemma}\label{lemma:exp} For $\mu>0$, $\beta\geq1$ and $\lambda\in\mathbb R$, we have
$$\DS\LD \left((t^\r-a^\r)^{\beta-1}E_{\mu,\beta}\left[\lambda(t^\r-a^\r)^\mu\right]\right)=\DS\left\{
\begin{array}{ll}
\DS\r^\a(t^\r-a^\r)^{\beta-\a-1}E_{\mu,\beta-\a}\left[\lambda(t^\r-a^\r)^\mu\right]& \mbox{ if } \quad \beta>1\\
\DS\lambda\r^\a(t^\r-a^\r)^{\mu-\a}E_{\mu,\mu-\a+1}\left[\lambda(t^\r-a^\r)^\mu\right]& \mbox{ if } \quad \beta=1\\
\end{array} \right. $$
and
$$
\DS\RD \left((b^\r-t^\r)^{\beta-1}E_{\mu,\beta}\left[\lambda(b^\r-t^\r)^\mu\right]\right)=\DS\left\{
\begin{array}{ll}
\DS\r^\a(b^\r-t^\r)^{\beta-\a-1}E_{\mu,\beta-\a}\left[\lambda(b^\r-t^\r)^\mu\right]& \mbox{ if } \, \beta>1\\
\DS\lambda\r^\a(b^\r-t^\r)^{\mu-\a}E_{\mu,\mu-\a+1}\left[\lambda(b^\r-t^\r)^\mu\right]& \mbox{ if } \, \beta=1\\
\end{array} \right.. $$
\end{lemma}

\begin{proof} Let $\beta>1$. Using Lemma \ref{lemma:power}, we have the following
$$\begin{array}{ll}
\DS\LD \left((t^\r-a^\r)^{\beta-1}E_{\mu,\beta}\left[\lambda(t^\r-a^\r)^\mu\right]\right)&
 = \DS\sum_{k=0}^\infty\frac{\lambda^k}{\Gamma(\mu k+\beta)}\LD (t^\r-a^\r)^{\mu k+\beta-1}\\
& = \DS\sum_{k=0}^\infty\frac{\lambda^k}{\Gamma(\mu k+\beta)}\frac{\r^\a\Gamma(\mu k+\beta)}{\Gamma(\mu k+\beta-\a)}(t^\r-a^\r)^{\mu k+\beta-\a-1}\\
&=\DS\r^\a(t^\r-a^\r)^{\beta-\a-1}E_{\mu,\beta-\a}\left[\lambda(t^\r-a^\r)^\mu\right].\\
\end{array}$$
For $\beta=1$, we have that
$$\begin{array}{ll}
\DS\LD \left(E_{\mu,1}\left[\lambda(t^\r-a^\r)^\mu\right]\right)&
 = \DS\LD 1+\sum_{k=1}^\infty\frac{\lambda^k}{\Gamma(\mu k+1)}\LD (t^\r-a^\r)^{\mu k}\\
& = \DS\sum_{k=1}^\infty\frac{\lambda^k}{\Gamma(\mu k+1)}\frac{\r^\a\Gamma(\mu k+1)}{\Gamma(\mu k-\a+1)}(t^\r-a^\r)^{\mu k-\a}\\
&=\DS\lambda\r^\a(t^\r-a^\r)^{\mu-\a}\sum_{k=0}^\infty\frac{\lambda^k}{\Gamma(\mu k+\mu-\a+1)}(t^\r-a^\r)^{\mu k}\\
&=\DS\lambda\r^\a(t^\r-a^\r)^{\mu-\a}E_{\mu,\mu-\a+1}\left[\lambda(t^\r-a^\r)^\mu\right].
\end{array}$$
The formula for the right fractional derivatives can be proven in a similar way.
\end{proof}

In particular, setting $\mu=\beta=1$, we get the formulas for the fractional derivative of the exponential functions:
$$\LD \left(\exp\left[\lambda(t^\r-a^\r)\right]\right)=\lambda\r^\a(t^\r-a^\r)^{1-\a}E_{1,2-\a}\left[\lambda(t^\r-a^\r)\right]$$
and
$$\RD \left(\exp\left[\lambda(b^\r-t^\r)\right]\right)=\lambda\r^\a(b^\r-t^\r)^{1-\a}E_{1,2-\a}\left[\lambda(b^\r-t^\r)\right].$$
Observe that, for $\r=1$ and $\a\to1^-$, we obtain the usual derivative of the exponential function:
$$\frac{d}{dt} \left(\exp\left[\lambda(t-a)\right]\right)=\lambda\exp\left[\lambda(t-a)\right].$$
For $\alpha=\mu>0$ and $\lambda\in\mathbb R$, we have
$$\DS\LD \left(E_\a\left[\lambda(t^\r-a^\r)^\a\right]\right)=\lambda\r^\a E_\a\left[\lambda(t^\r-a^\r)^\a\right]$$
and
$$\RD \left(E_\a\left[\lambda(b^\r-t^\r)^\a\right]\right)=\lambda\r^\a E_\a\left[\lambda(b^\r-t^\r)^\a\right].$$

\section{Caputo--Katugampola fractional Cauchy problem}\label{sec:FDE}

Consider the nonlinear fractional differential equation of order $\a\in(0,1)$,
\begin{equation}\label{eq:CauchyEq}
\LD x(t)=f(t,x(t)), \quad t\in[a,b],
\end{equation}
with the initial condition
\begin{equation}\label{eq:BoundCond}
x(a)=x_a, \quad x_a\in\mathbb R,
\end{equation}
where $f:[a,b]\times\mathbb R\to \mathbb R$ is a continuous function with respect to all its arguments.
We seek conditions that guarantee the existence and uniqueness of solution to problem \eqref{eq:CauchyEq}--\eqref{eq:BoundCond}
in the following set of functions
$$\mathcal{U}:=\left\{x\in C([a,b]):\LD x\in C([a,b])\right\}.$$

First, let us observe that, for $x\in C([a,b])$, the Cauchy type problem \eqref{eq:CauchyEq}--\eqref{eq:BoundCond} is equivalent to the problem
of finding solutions to the following Volterra integral equation
\begin{equation}\label{eq:Volterra}
x(t)=x_a+\frac{\rho^{1-\alpha}}{\Gamma(\alpha)}\int_a^t\frac{\tau^{\rho-1}}{(t^{\rho}-\tau^{\rho})^{1-\alpha}}f(\t,x(\t))\; d\tau.
\end{equation}

Indeed, if $x\in C([a,b])$ satisfies \eqref{eq:CauchyEq}--\eqref{eq:BoundCond}, then applying operator $\LI$ to the both
sides of \eqref{eq:CauchyEq}, using the initial condition \eqref{eq:BoundCond} and applying Theorem~\ref{teo:ID}
we obtain equation \eqref{eq:Volterra}. Conversely, taking { $t\rightarrow a^+$}, we see that the initial condition
\eqref{eq:BoundCond} is satisfied and applying operator $\LD$ to both sides of \eqref{eq:Volterra}, by
Theorem~\ref{thm:DerInt}, we arrive to problem \eqref{eq:CauchyEq}--\eqref{eq:BoundCond}.

Now we will use the above results to prove the following theorem, ensuring the existence and uniqueness of solution
to the Cauchy type problem \eqref{eq:CauchyEq}--\eqref{eq:BoundCond}.

\begin{theorem}
Let $f:[a,b]\times\mathbb R\to \mathbb R$ be a continuous function and Lipschitz with respect to the second variable, i.e.,
\begin{equation}\label{eq:LipType}
\left|f(t,x_1)-f(t,x_2)\right|<L\left|x_1-x_2\right|,~(L>0),
\end{equation}
for all $t\in [a,b]$ and all $x_1,x_2\in\mathbb{R}$. Then, the Cauchy type problem \eqref{eq:CauchyEq}--\eqref{eq:BoundCond}
possesses a unique solution in the space $\mathcal{U}$.
\end{theorem}

\begin{proof}
The proof is similar in spirit to \cite{Kilbas2,Kilbas}. We start by showing that, for \eqref{eq:CauchyEq}--\eqref{eq:BoundCond}, there exists a unique
solution $x\in C([a,b])$. Let us recall that
the Cauchy type problem \eqref{eq:CauchyEq}--\eqref{eq:BoundCond} is equivalent to the problem
of finding solutions to the Volterra integral equation \eqref{eq:Volterra}.
This allows us to use the well known method, for nonlinear Volterra integral equations,
where first we prove existence and uniqueness of solutions on a subinterval of $[a,b]$.\\
Let us choose $a<t_1<b$ to be such that the following condition is satisfied
\begin{equation}\label{eq:t1}
0<L\frac{\rho^{-\alpha}}{\Gamma(\alpha+1)}(t_1^{\rho}-a^{\rho})^{\alpha}<1.
\end{equation}
We shall prove the existence of a unique solution $x\in C([a,t_1])$ to \eqref{eq:Volterra} on the subinterval $[a,t_1]\subset[a,b]$.
Let us define the following integral operator:
\begin{equation}
\fonction{\mathcal{S}}{C([a,b])}{C([a,b])}{x(t)}{\mathcal{S}x(t):=\displaystyle x_a+\frac{\rho^{1-\alpha}}{\Gamma(\alpha)}\int_a^t\frac{\tau^{\rho-1}}{(t^{\rho}-\tau^{\rho})^{1-\alpha}}f(\t,x(\t))\; d\tau,}
\end{equation}
and rewrite \eqref{eq:Volterra} in the form $x(t)=\mathcal{S}x(t)$. Note that $\mathcal{S}$ is well defined and is a  bounded operator on the complete metric space $C([a,b])$ (by Theorem~\ref{lem:contInt}).
Using Theorem~\ref{lem:contInt} and the condition \eqref{eq:LipType} one has
\begin{multline*}
\left\|\mathcal{S}x_1-\mathcal{S}x_2\right\|_{C([a,t_1])}
\leq\left\|\LI\left( f(t,x_1)-f(t,x_2)\right)\right\|_{C([a,t_1])}\\
\leq L\left\|\LI\left(x_1-x_2\right)\right\|_{C([a,t_1])}
\leq L\frac{\rho^{-\alpha}}{\Gamma(\alpha+1)}(t_1^{\rho}-a^{\rho})^{\alpha}\left\|x_1-x_2\right\|_{C([a,t_1])}
\end{multline*}
and because condition \eqref{eq:t1} is satisfied, by the Banach fixed point theorem, there exists a unique solution $x^{*1}\in C([a,t_1])$ to the equation \eqref{eq:Volterra}
on the interval $[a,t_1]$. Moreover, if we define the sequence $x_m^1(t):=\mathcal{S}^m x_0(t)$ where, for $x_0(t)=x_a$,
\begin{equation*}
\mathcal{S}^m x_0(t):=x_0(t)+\frac{\rho^{1-\alpha}}{\Gamma(\alpha)}\int_a^t\frac{\tau^{\rho-1}}{(t^{\rho}-\tau^{\rho})^{1-\alpha}}f(\t,\mathcal{S}^{m-1} x_0(\t))\; d\tau,~m=1,2,\dots,
\end{equation*}
then, again, by the Banach fixed point theorem, we obtain the solution $x^{*1}$ as a limit of the sequence $x_m^1$, i.e.,
\begin{equation*}
\lim\limits_{m\rightarrow\infty}\left\|x_m^1-x^{*1}\right\|_{C([a,t_1])}=0.
\end{equation*}
Now, let us choose $t_2=t_1+h_1$, with $h_1>0$ such that $t_2<b$ and
\begin{equation}\label{eq:t2}
0<L\frac{\rho^{-\alpha}}{\Gamma(\alpha+1)}(t_2^{\rho}-t_1^{\rho})^{\alpha}<1.
\end{equation}
Consider the interval $[t_1,t_2]$ and write equation \eqref{eq:Volterra} in the form
\begin{equation}\label{eq:Volterra2}
x(t)=x_0(t)+\frac{\rho^{1-\alpha}}{\Gamma(\alpha)}\int_{t_1}^t\frac{\tau^{\rho-1}}{(t^{\rho}-\tau^{\rho})^{1-\alpha}}f(\t,x(\t))\; d\tau
+\frac{\rho^{1-\alpha}}{\Gamma(\alpha)}\int_a^{t_1}\frac{\tau^{\rho-1}}{(t^{\rho}-\tau^{\rho})^{1-\alpha}}f(\t,x(\t))\; d\tau.
\end{equation}
Because on the interval $[a,t_1]$, equation \eqref{eq:Volterra2} possesses a unique solution, we can rewrite \eqref{eq:Volterra2}
as follows
\begin{equation*}\label{eq:Volterra3}
x(t)=x_{01}(t)+\frac{\rho^{1-\alpha}}{\Gamma(\alpha)}\int_{t_1}^{t}\frac{\tau^{\rho-1}}{(t^{\rho}-\tau^{\rho})^{1-\alpha}}f(\t,x(\t))\; d\tau,
\end{equation*}
where
\begin{equation*}
x_{01}(t)=x_{0}(t)+\frac{\rho^{1-\alpha}}{\Gamma(\alpha)}\int_{a}^{t_1}\frac{\tau^{\rho-1}}{(t^{\rho}-\tau^{\rho})^{1-\alpha}}f(\t,x(\t))\; d\tau
\end{equation*}
is the known function. By the same argument as before, we prove that there exists a unique solution $x^{*2}\in C([t_1,t_2])$ to equation \eqref{eq:Volterra}
on $[t_1,t_2]$. Repeating the previous reasoning, choosing $t_{k}=t_{k-1}+h_{k-1}$ such that $h_{k-1}>0$, $t_k<b$ and
\begin{equation}\label{eq:tk}
0<L\frac{\rho^{-\alpha}}{\Gamma(\alpha+1)}(t_{k}^{\rho}-t_{k-1}^{\rho})^{\alpha}<1,
\end{equation}
we see that equation \eqref{eq:Volterra} possesses a solution
$x^{*k}\in C([x_{k-1},x_k])$ on each interval $[x_{k-1},x_k]$ ($k=1,\dots,l$), where $a=x_0<x_1<\dots<x_l=b$ and we conclude that for problem \eqref{eq:CauchyEq}-\eqref{eq:BoundCond}, there exists a unique solution $x\in C([a,b])$.\\
It remains to prove that $x\in\mathcal{U}$, i.e., we need to show that $\LD x\in C([a,b])$. Recall that our solution $x$ can be approximated
by the sequence $x_m(t)=\mathcal{S}^m x_0(t)$, i.e.,
\begin{equation*}
\lim\limits_{m\rightarrow\infty}\left\|x_m-x\right\|_{C([a,b])}=0,
\end{equation*}
with the choice of certain $x_m$ on each interval $[a,t_1],\dots,[t_{l-1},b]$.
Using \eqref{eq:CauchyEq} and the Lipschitz type condition \eqref{eq:LipType} we have
\begin{equation*}
\left\|\LD x_m-\LD x\right\|_{C([a,b])} =\left\|f(t,x_m)-f(t,x)\right\|_{C([a,b])}
\leq L\left\|x_m-x\right\|_{C([a,b])},
\end{equation*}
then, taking $m\rightarrow\infty$, one has
\begin{equation*}
\lim\limits_{m\rightarrow\infty}\left\|\LD x_m-\LD x\right\|_{C([a,b])}=0.
\end{equation*}
Since $\LD x_m(t)=f(t,x_m(t))$ is continuous on $[a,b]$ we have that $\LD x$ belongs to the space $C([a,b])$. The proof is complete.
\end{proof}

\section{Numerical approach to the Caputo--Katugampola fractional Cauchy problem}\label{sec:approx}

In this section we prove an approximation formula for the Caputo--Katugampola fractional derivative, that depends only on the first-order derivative of a function $x$. With this tool in hand, for a given Caputo--Katugampola fractional Cauchy problem we can consider a sequence of systems of $N+1$ ordinary differential equations with $N+1$ initial conditions, where $N\in \N$. A sequence $(x_N)$ of the solutions to those systems converges to the solution of the fractional problem.

\begin{theorem}\label{teo:approx} Let $x:[a,b]\rightarrow\mathbb{R}$ be a function of class $C^2$, and $N \geq 1$ an integer. Define the quantities
$$\begin{array}{ll}
A_N&=\DS\frac{\r^{\a-1}}{\Gamma(2-\a)}\sum_{k=0}^N\frac{\Gamma(k-1+\a)}{\Gamma(\a-1)k!},\\
B_{N,k}&=\DS\frac{\r^\a\Gamma(k-1+\a)}{\Gamma(2-\a)\Gamma(\a-1)(k-1)!},\quad  k=1,\ldots,N,\\
\end{array}$$
and functions $V_k:[a,b]\rightarrow\mathbb{R}$ by
$$V_k(t)=\int_a^t (\t^\r-a^\r)^{k-1}x'(\t)d\t,\quad k=1,\ldots,N.$$
Then,
\begin{equation}\label{eq:mainformula}\LD x(t)=A_N(t^\r-a^\r)^{1-\a} t^{1-\r}x'(t)-\sum_{k=1}^N B_{N,k}(t^\r-a^\r)^{1-\a-k}V_k(t)+E_N(t),\end{equation}
with
$$\lim_{N\to\infty}E_N(t)=0, \quad \forall t\in[a,b].$$
\end{theorem}

\begin{proof} Set
$$u'(\t)=(t^\r-\t^\r)^{-\a}\t^{\r-1} \quad \mbox{and} \quad v(\t)=\t^{1-\r}x'(\t).$$
Then, integrating by parts
we obtain
$$\begin{array}{ll}
\DS\LD x(t) &=\DS\frac{\r^\a}{\Gamma(1-\alpha)}\int_a^t (t^\r-\t^\r)^{-\a}x'(\t) d\tau\\
&=\DS\frac{\r^{\a-1}}{\Gamma(2-\alpha)}(t^\r-a^\r)^{1-\a}a^{1-\r}x'(a)+\frac{\r^{\a-1}}{\Gamma(2-\alpha)}\int_a^t (t^\r-\t^\r)^{1-\a}
\frac{d}{d\t}\left(\t^{1-\r}x'(\t)\right) d\tau.\\
\end{array}$$
Using the generalized binomial theorem, we get the decomposition
$$\begin{array}{ll}
(t^\r-\tau^\r)^{1-\alpha}& =\DS (t^\r-a^\r)^{1-\alpha}\left(1-\frac{\t^\r-a^\r}{t^\r-a^\r}\right)^{1-\alpha}\\
& =\DS (t^\r-a^\r)^{1-\alpha}\sum_{k=0}^\infty \frac{\Gamma(k-1+\a)}{\Gamma(\a-1)k!}\left(\frac{\t^\r-a^\r}{t^\r-a^\r}\right)^k.
\end{array}$$
Substituting this series into the expression of the fractional derivative, we obtain
$$\begin{array}{ll}
\DS\LD x(t) &=\DS \frac{\r^{\a-1}}{\Gamma(2-\alpha)}(t^\r-a^\r)^{1-\a}a^{1-\r}x'(a)+\frac{\r^{\a-1}}{\Gamma(2-\alpha)}\int_a^t (t^\r-a^\r)^{1-\a}\\
&\quad \DS\times \sum_{k=0}^N \frac{\Gamma(k-1+\a)}{\Gamma(\a-1)k!}\left(\frac{\t^\r-a^\r}{t^\r-a^\r}\right)^k \frac{d}{d\t}\left(\t^{1-\r}x'(\t)\right) d\tau+E_N(t),\\
\end{array}$$
where
$$E_{N}(t)=\frac{\r^{\a-1}}{\Gamma(2-\alpha)}\int_a^t (t^\r-a^\r)^{1-\a}\sum_{k=N+1}^\infty \frac{\Gamma(k-1+\a)}{\Gamma(\a-1)k!}\left(\frac{\t^\r-a^\r}{t^\r-a^\r}\right)^k \frac{d}{d\t}\left(\t^{1-\r}x'(\t)\right) d\tau.$$
Now, let us split the sum into $k=0$ and $k=1,\ldots, N$, that is,
$$\begin{array}{ll}
\DS\LD x(t) &=\DS \frac{\r^{\a-1}}{\Gamma(2-\alpha)}(t^\r-a^\r)^{1-\a}t^{1-\r}x'(t)+\frac{\r^{\a-1}}{\Gamma(2-\alpha)}(t^\r-a^\r)^{1-\a} \\
&\quad \DS\times \sum_{k=1}^N \frac{\Gamma(k-1+\a)}{\Gamma(\a-1)k!(t^\r-a^\r)^k}\int_a^t(\t^\r-a^\r)^k \frac{d}{d\t}\left(\t^{1-\r}x'(\t)\right) d\tau+E_N(t).\\
\end{array}$$
Setting
$$u(\t)=(\t^\r-a^\r)^k \quad \mbox{and} \quad v'(\t)=\frac{d}{d\t}\left(\t^{1-\r}x'(\t)\right)$$
and then integrating by parts, we obtain
$$\begin{array}{ll}
\DS\LD x(t) &=\DS \frac{\r^{\a-1}}{\Gamma(2-\alpha)}(t^\r-a^\r)^{1-\a}t^{1-\r}x'(t)\left[1+\sum_{k=1}^N\frac{\Gamma(k-1+\a)}{\Gamma(\a-1)k!}\right]\\
&\quad \DS-\frac{\r^{\a}}{\Gamma(2-\alpha)}\sum_{k=1}^N \frac{\Gamma(k-1+\a)}{\Gamma(\a-1)(k-1)!}(t^\r-a^\r)^{1-\a-k}
\int_a^t(\t^\r-a^\r)^{k-1}x'(\t)d\tau+E_N(t).\\
\end{array}$$
It remains to show that
$$\lim_{N\to\infty}E_N(t)=0, \quad \forall t\in[a,b],$$
and to do this we determine an upper bound for the error. Since $\t\in[a,t]$, we have
$$\sum_{k=N+1}^\infty \left|\frac{\Gamma(k-1+\a)}{\Gamma(\a-1)k!}\left(\frac{\t^\r-a^\r}{t^\r-a^\r}\right)^k \right|\leq \sum_{k=N+1}^\infty \frac{\exp((1-\a)^2+1-\a)}{k^{2-\a}}$$
$$\leq\int_N^\infty \frac{\exp((1-\a)^2+1-\a)}{k^{2-\a}}\, dk= \frac{\exp((1-\a)^2+1-\a)}{N^{1-\a}(1-\a)}.$$
Taking
$$M(t):=\max_{\t\in[a,t]}\left|\frac{d}{d\t}\left(\t^{1-\r}x'(\t)\right) \right|,$$
we get the following upper bound:
$$\begin{array}{ll}
\DS\left|E_N(t)\right|& \DS \leq M(t)\frac{\exp((1-\a)^2+1-\a)}{N^{1-\a}(1-\a)}\frac{\r^{\a-1}}{\Gamma(2-\alpha)}\int_a^t (t^\r-a^\r)^{1-\a}d\t\\
&\DS= M(t)\frac{\exp((1-\a)^2+1-\a)\r^{\a-1}}{N^{1-\a}(1-\a)\Gamma(2-\alpha)}(t^\r-a^\r)^{1-\a}(t-a),\\
\end{array}$$
which converges to zero as $N$ goes to $\infty$, for al $t\in[a,b]$.
\end{proof}

\begin{remark} We stress that, when $\r=1$, Theorem~\ref{teo:approx} reduces to the one proven in \cite{Atan1,Pooseh1}, and as $\r\to0^+$, we obtain the main result of \cite{Pooseh0}.
\end{remark}

In an analogous way we obtain a decomposition formula for the right Caputo--Katugampola fractional derivative. The proof is omitted, since it is similar to the proof of Theorem~\ref{teo:approx}.

\begin{theorem} Let $x:[a,b]\rightarrow\mathbb{R}$ be a function of class $C^2$, and $N \geq 1$ an integer. Define the quantities
$$\begin{array}{ll}
A_N&=\DS\frac{\r^{\a-1}}{\Gamma(2-\a)}\sum_{k=0}^N\frac{\Gamma(k-1+\a)}{\Gamma(\a-1)k!},\\
B_{N,k}&=\DS\frac{\r^\a\Gamma(k-1+\a)}{\Gamma(2-\a)\Gamma(\a-1)(k-1)!},\quad k=1,\ldots,N,\\
\end{array}$$
and functions $W_k:[a,b]\rightarrow\mathbb{R}$ by
$$W_k(t)=\int_t^b (b^\r-\t^\r)^{k-1}x'(\t)d\t,\quad  k=1,\ldots,N.$$
Then,
$$\RD x(t)=-A_N(b^\r-t^\r)^{1-\a} t^{1-\r}x'(t)+\sum_{k=1}^N B_{N,k}(b^\r-t^\r)^{1-\a-k}W_k(t)+E_N(t),$$
such that
$$\lim_{N\to\infty}E_N(t)=0, \quad \forall t\in[a,b].$$
\end{theorem}

{ Due to } the complexity of solving analytically fractional differential equations, numerical tools are often applied to solve the given problems (see e.g. \cite{ford1,ford2,Gracia,Sousa,Yan,Yang}).
We now present a simple but efficient method to solve a fractional differential equation as presented in Section~\ref{sec:FDE}. The idea is to replace the fractional derivative $\LD x$ in the differential equation by an approximation given in Theorem~\ref{teo:approx} up to order $N\in\mathbb N$. In this way we obtain a system of $N+1$ ODE's, with $N+1$ initial conditions.
{ The advantage of this procedure is that after replacing the fractional derivative by the expansion which involves ordinary derivatives only, we no longer deal with a fractional differential equation but with a system of ordinary differential equations. With this, we can apply any known technique from the literature (analytical or numerical) to solve the problem.}
The solution to this system $x_N$ obviously depends on $N$. However, the sequence $(x_N)$ converges to the solution of the fractional differential equation in a neighborhood of the initial point. We refer to \cite{Atan2}, where similar problems, with dependence of Caputo and Riemann-Liouville fractional derivatives, are addressed.

Recall that the Cauchy type problem \eqref{eq:CauchyEq}--\eqref{eq:BoundCond} is equivalent to the problem of finding solutions to the Volterra integral equation \eqref{eq:Volterra}. This last equation, using \eqref{Derivative_Integral}, can be rewritten in the following way:
$$x(t)=x(a)+{^CD_{a+}^{1-\a,\r}}\left[\int_a^t f(\t,x(\t)) \t^{\r-1}\,d\t\right].$$
 Applying the decomposition formula given in Theorem~\ref{teo:approx}, we obtain
\begin{equation}\label{solExact}
x(t)=x(a)+A_N(t^\r-a^\r)^{\a}f(t,x(t))-\sum_{k=1}^N B_{N,k}(t^\r-a^\r)^{\a-k}V_k(t)+E_N(t),
\end{equation}
where
$$\begin{array}{ll}
A_N&=\DS\frac{\r^{-\a}}{\Gamma(1+\a)}\sum_{k=0}^N\frac{\Gamma(k-\a)}{\Gamma(-\a)k!},\\
B_{N,k}&=\DS\frac{\r^{1-\a}\Gamma(k-\a)}{\Gamma(1+\a)\Gamma(-\a)(k-1)!},\quad \mbox{for} \, k \in \{1,\ldots,N\},\\
V_k(t)&=\displaystyle\int_a^t (\t^\r-a^\r)^{k-1}f(\t,x(\t))\t^{\r-1}d\t,\quad \mbox{for} \, k \in \{1,\ldots,N\},
\end{array}$$
and the error is bounded by the formula
$$\left|E_N(t)\right|\DS \leq M(t)\frac{\exp(\a^2+\a)\r^{-\a}}{\a N^{\a}\Gamma(1+\alpha)}(t^\r-a^\r)^{\a}(t-a),$$
with
$$M(t):=\max_{\t\in[a,t]}\left|\frac{d}{d\t}f(\t,x(\t))\right|.$$
In order to get an approximated solution, $x_N$, we truncate the formula up to order $N$, receiving
\begin{equation}\label{solApprox}
x_N(t)=x(a)+A_N(t^\r-a^\r)^{\a}f(t,x_N(t))-\sum_{k=1}^N B_{N,k}(t^\r-a^\r)^{\a-k}V_{k,N}(t)
\end{equation}
where
$$V_{k,N}(t)=\int_a^t (\t^\r-a^\r)^{k-1}f(\t,x_N(\t))\t^{\r-1}d\t,\quad \mbox{for} \, k \in \{1,\ldots,N\},$$
(we remark that $x_N(a)=x(a))$.

\begin{theorem} Let $f:[a,b]\times \mathbb R \to \mathbb R$ be a continuous function and Lipschitz with respect to the second variable (see \eqref{eq:LipType}).
For $N\in\mathbb N$, let $x$ and $x_N$ be given by conditions \eqref{solExact} and \eqref{solApprox}, respectively. Also, let $T\in\mathbb R$ be a real in the open interval
$$a<T<\left(a^\r+\r\left(\frac{\Gamma(1+\a)}{L}\right)^{\frac1\a}\right)^{\frac1\r}.$$
Then, for all $t\in[a,T]$, $x_N(t)\to x(t)$ as $N\to\infty$.
\end{theorem}

\begin{proof} Starting with relations \eqref{solExact} and \eqref{solApprox} we have
$$|x_N(t)-x(t)|\leq |A_N|(t^\r-a^\r)^{\a}|f(t,x_N(t))-f(t,x(t))|+\sum_{k=1}^N |B_{N,k}|(t^\r-a^\r)^{\a-k}|V_{k,N}(t)-V_k(t)|+|E_N(t)|$$
for all $t\in[a,T]$.
Define
$$\delta x_N:=\max_{t\in[a,T]}|x_N(t)-x(t)|.$$
As a consequence of the following relations:
$$|f(t,x_N(t))-f(t,x(t))|\leq L|x_N(t)-x(t)|\leq L \delta x_N;$$
$$\begin{array}{ll}
|V_{k,N}(t)-V_k(t)|&=\left|\displaystyle \int_a^t (\t^\r-a^\r)^{k-1}|f(\t,x_N(\t))-f(\t,x(\t))|\t^{\r-1}d\t\right|\\
&\displaystyle\leq L \delta x_N\int_a^t (\t^\r-a^\r)^{k-1}\t^{\r-1}d\t=\frac{L \delta x_N}{k\r}(t^\r-a^\r)^k;
\end{array}$$
$$\begin{array}{ll}
    |A_N|&=\displaystyle\frac{\r^{-\a}}{\Gamma(1+\a)}\left|\sum_{k=0}^N\frac{\Gamma(k-\a)}{\Gamma(-\a)k!}\right|=\frac{\r^{-\a}}{\Gamma(1+\a)|\Gamma(-\a)|}\cdot \frac{\Gamma(N+1-\a)}{\a \Gamma(N+1)}\\
&    \displaystyle =\frac{\r^{-\a}}{\Gamma(1+\a)\Gamma(1-\a)}\cdot \frac{\Gamma(N+1-\a)}{\Gamma(N+1)}
  \leq\frac{1}{\r^\a\a\pi}\cdot \frac{\Gamma(N+1-\a)}{\Gamma(N+1)}\end{array},$$
by the  Euler's reflection formula and equation (3) in \cite{Garrappa},
 $$\begin{array}{ll}
 \displaystyle\sum_{k=1}^N |B_{N,k}|(t^\r-a^\r)^{\a-k}|V_{k,N}(t)-V_k(t)|&\displaystyle\leq \frac{L \delta x_N(t^\r-a^\r)^\a }{\r^\a\Gamma(1+\a)|\Gamma(-\a)|} \sum_{k=1}^N\frac{\Gamma(k-\a)}{k!}\\
&\displaystyle\leq  \frac{L \delta x_N(t^\r-a^\r)^\a}{\r^\a\Gamma(1+\a)|\Gamma(-\a)|} \left[\frac{\Gamma(N+1-\a)}{\a\Gamma(N+1)}+|\Gamma(-\a)|\right]\\
&\displaystyle \leq  \frac{L \delta x_N}{\r^\a}(t^\r-a^\r)^\a \left[\frac{1}{\a\pi}\cdot \frac{\Gamma(N+1-\a)}{\Gamma(N+1)}+\frac{1}{\Gamma(1+\a)}\right],
\end{array}$$
we conclude that
$$|x_N(t)-x(t)|\leq \frac{L \delta x_N}{\r^\a}(t^\r-a^\r)^\a \left[\frac{2}{\a\pi}\cdot \frac{\Gamma(N+1-\a)}{\Gamma(N+1)}+\frac{1}{\Gamma(1+\a)}\right]+|E_N(t)|$$
for all $t\in[a,T]$.
Taking the maximum, over $t\in[a,T]$, on both sides of the inequality, we get
\begin{equation}\label{relation1}\delta x_N\leq \frac{L \delta x_N}{\r^\a}(T^\r-a^\r)^\a \left[\frac{2}{\a\pi}\cdot \frac{\Gamma(N+1-\a)}{\Gamma(N+1)}+\frac{1}{\Gamma(1+\a)}\right]+\max_{t\in[a,T]}|E_N(t)|.\end{equation}
It is obvious that
$$\lim_{N\to\infty}|E_N(t)|=0.$$
Moreover, as a consequence of the Stirling's formula (see e.g. \cite{Tricomi}), we have
$$\lim_{N\to\infty}\frac{\Gamma(N+1-\a)}{\Gamma(N+1)}=0.$$
Therefore, setting  $N\to\infty$ in \eqref{relation1} we obtain
$$\lim_{N\to\infty}\delta x_N\left[1-\frac{L}{\r^\a\Gamma(1+\a)}(T^\r-a^\r)^\a\right]\leq0,$$
and by the definition of $T$, we must have $\delta x_N\to0$.
\end{proof}

{ From Eq. \eqref{relation1} we obtain that
$$\delta x_N\left[1-\frac{L}{\r^\a}(T^\r-a^\r)^\a \left[\frac{2}{\a\pi}\cdot \frac{\Gamma(N+1-\a)}{\Gamma(N+1)}+\frac{1}{\Gamma(1+\a)}\right]\right]\leq\max_{t\in[a,T]}|E_N(t)|.$$
Since
$$0<\frac{L}{\r^\a}(T^\r-a^\r)^\a<\Gamma(1+\a),$$
we have
$$-\frac{2\Gamma(1+\a)\Gamma(N+1-\a)}{\a\pi\Gamma(N+1)}\leq1-\frac{L}{\r^\a}(T^\r-a^\r)^\a \left[\frac{2}{\a\pi}\cdot \frac{\Gamma(N+1-\a)}{\Gamma(N+1)}+\frac{1}{\Gamma(1+\a)}\right]<1.$$
Therefore, for $N$ sufficiently large, by
$$\lim_{N\to\infty}\frac{\Gamma(N+1-\a)}{\Gamma(N+1)}=0,$$
we get that there exists a function $C$ (depending on $\a,\r,t$, but independent of $N$) such that
$$\delta x_N \leq C N^{\a-1}.$$}

\section{Examples}\label{Sec:examples}

In this section we provide two examples of applications of the results presented in Section~\ref{sec:approx}.

\begin{example}

Let $x(t)=(t^\r-1)^2$, with $t\in[1,2]$, be a test function. By Lemma~\ref{lemma:power}, we have that
$${^CD_{1+}^{\a,\r}} x(t)=\frac{2\r^\a}{\Gamma(3-\a)}(t^\r-1)^{2-\a}.$$
In Figures \ref{fig:1}, \ref{fig:2} and \ref{fig:3} (on the left side) the exact fractional derivative of the function (continuous line) with different numerical approximations of the fractional derivative (dot lines), for different values of $\a\in(0,1)$, $\r>0$ and $N\in\mathbb N$, are compared. The error between the exact and the numerical approximations is measured by the absolute value of the difference between these two expressions, and it is shown on the right side of figures.

\begin{figure}[h!]
\begin{center}
\subfigure{\includegraphics[scale=0.3]{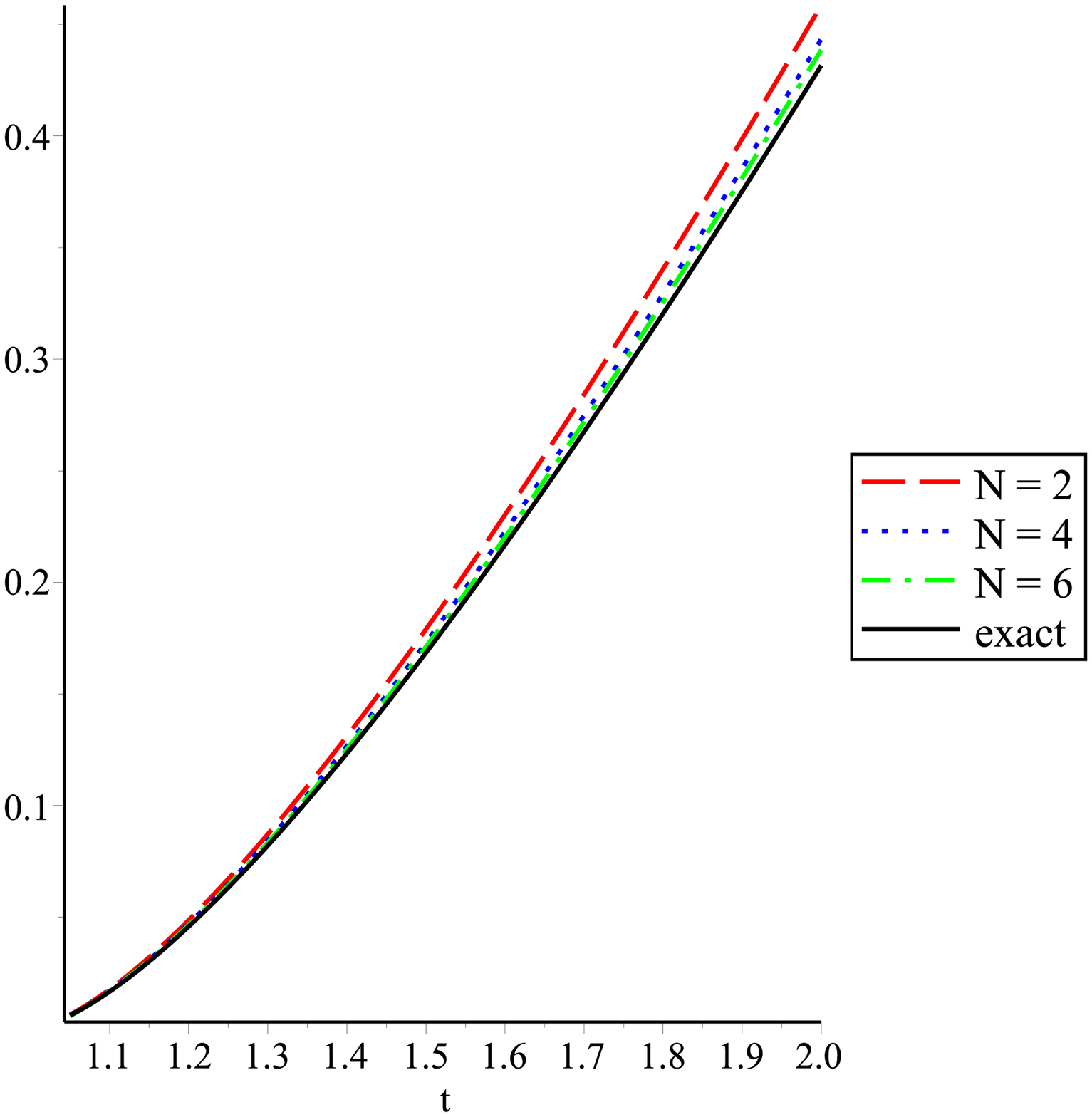}}
\subfigure{\includegraphics[scale=0.3]{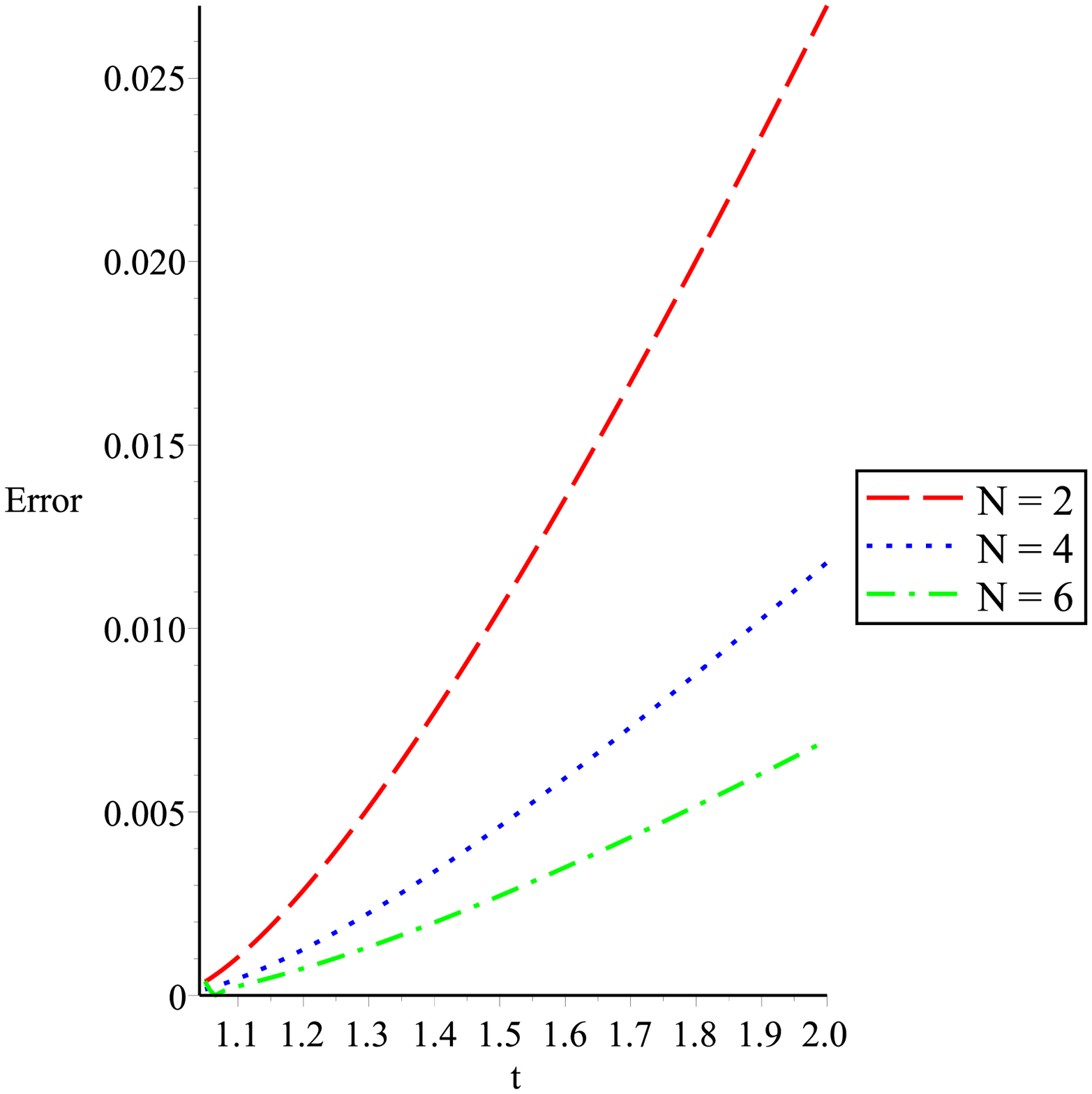}}
\caption{For $\a=0.5$ and $\r=0.6$.}\label{fig:1}
\end{center}
\end{figure}

\begin{figure}[h!]
\begin{center}
\subfigure{\includegraphics[scale=0.3]{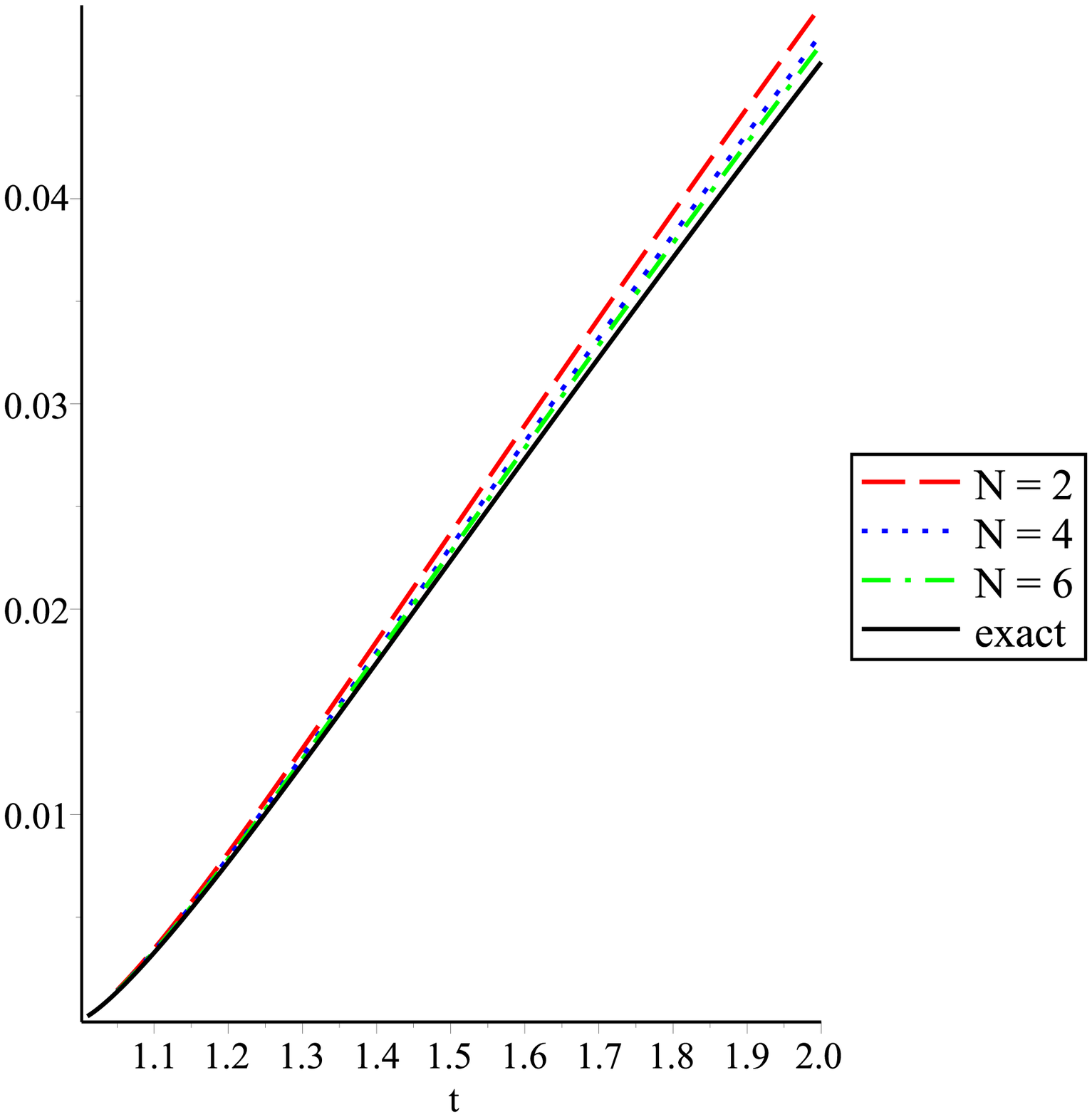}}
\subfigure{\includegraphics[scale=0.3]{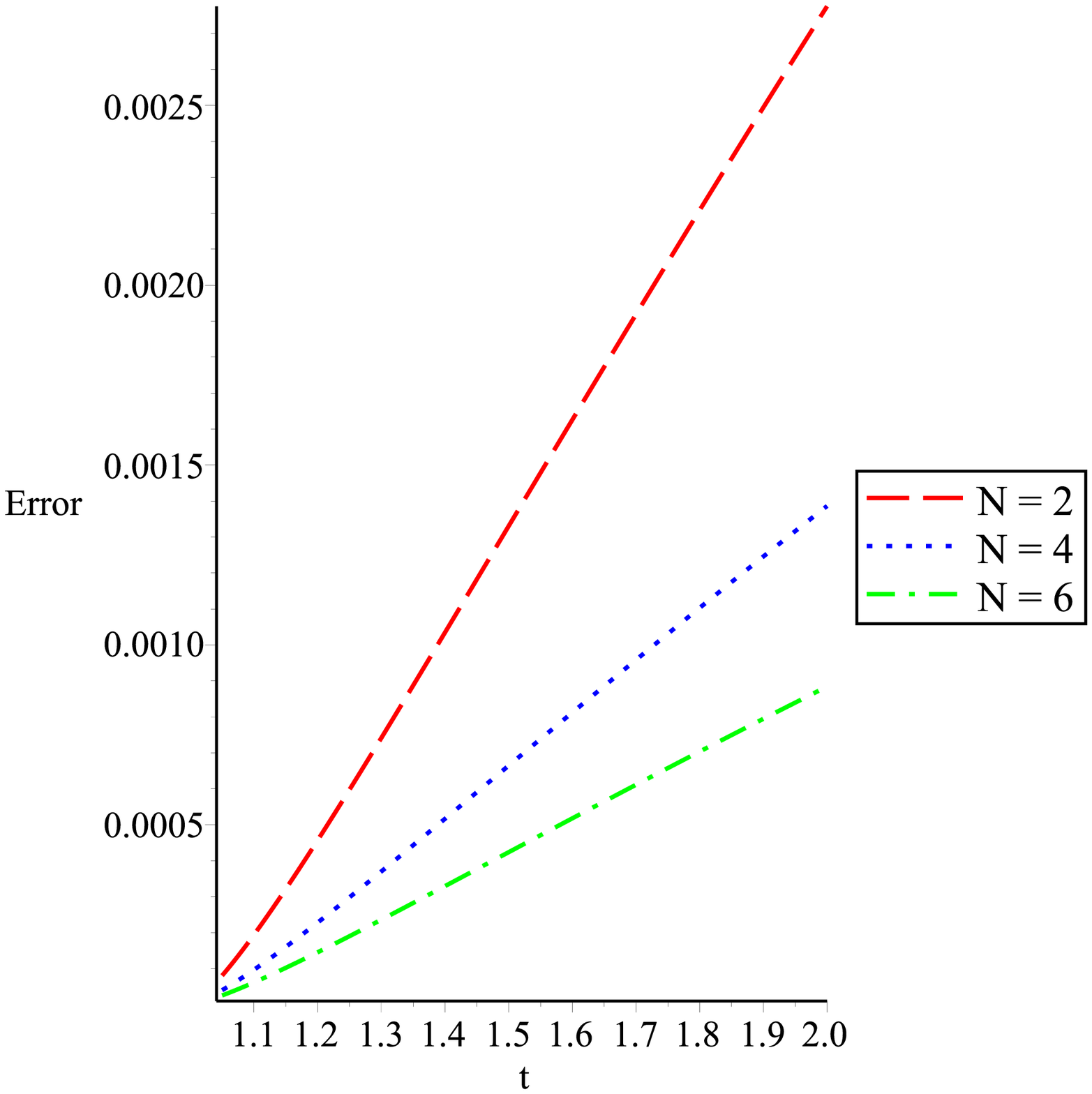}}
\caption{For $\a=0.7$ and $\r=0.2$.}\label{fig:2}
\end{center}
\end{figure}

\begin{figure}[h!]
\begin{center}
\subfigure{\includegraphics[scale=0.3]{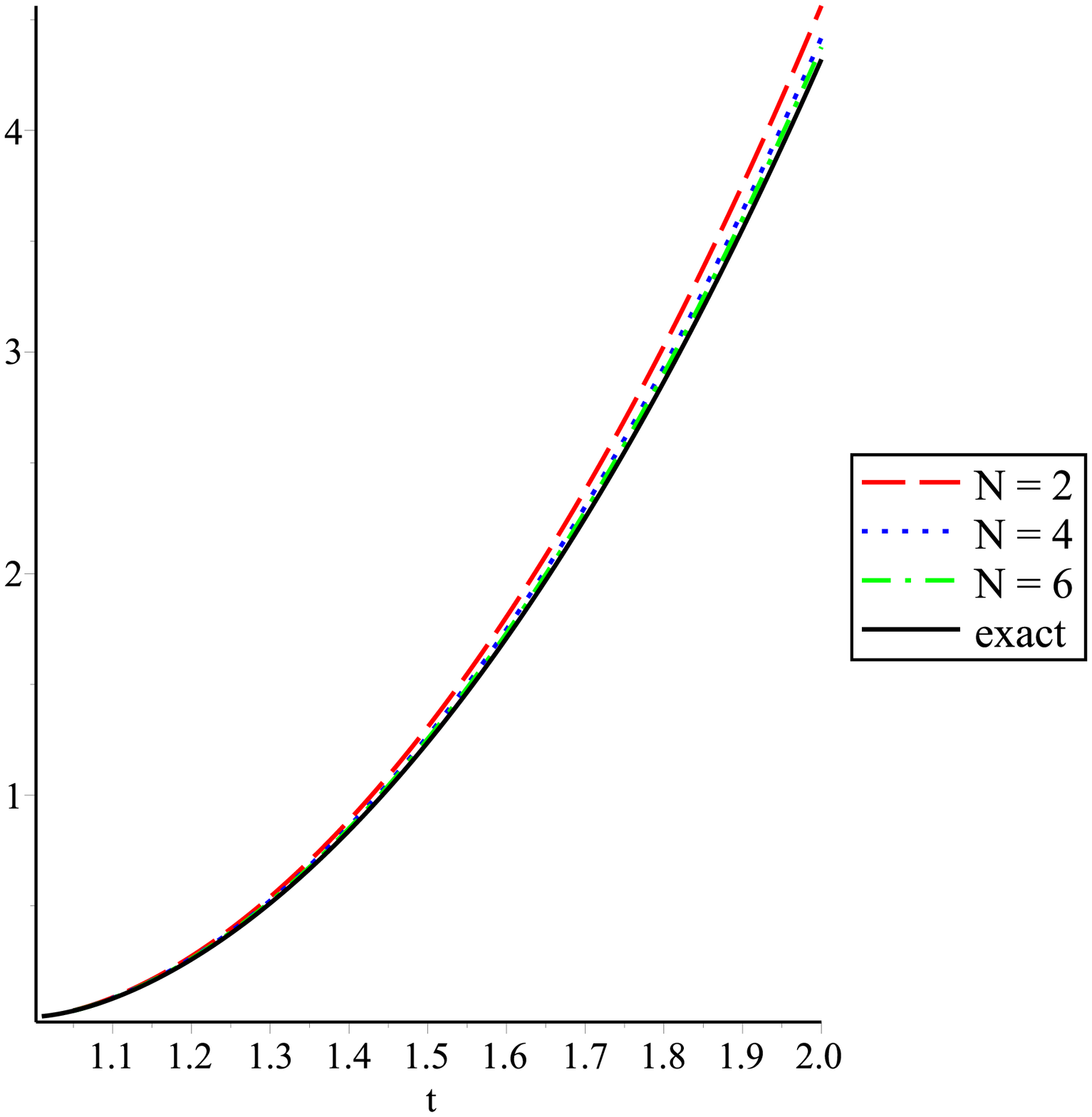}}
\subfigure{\includegraphics[scale=0.3]{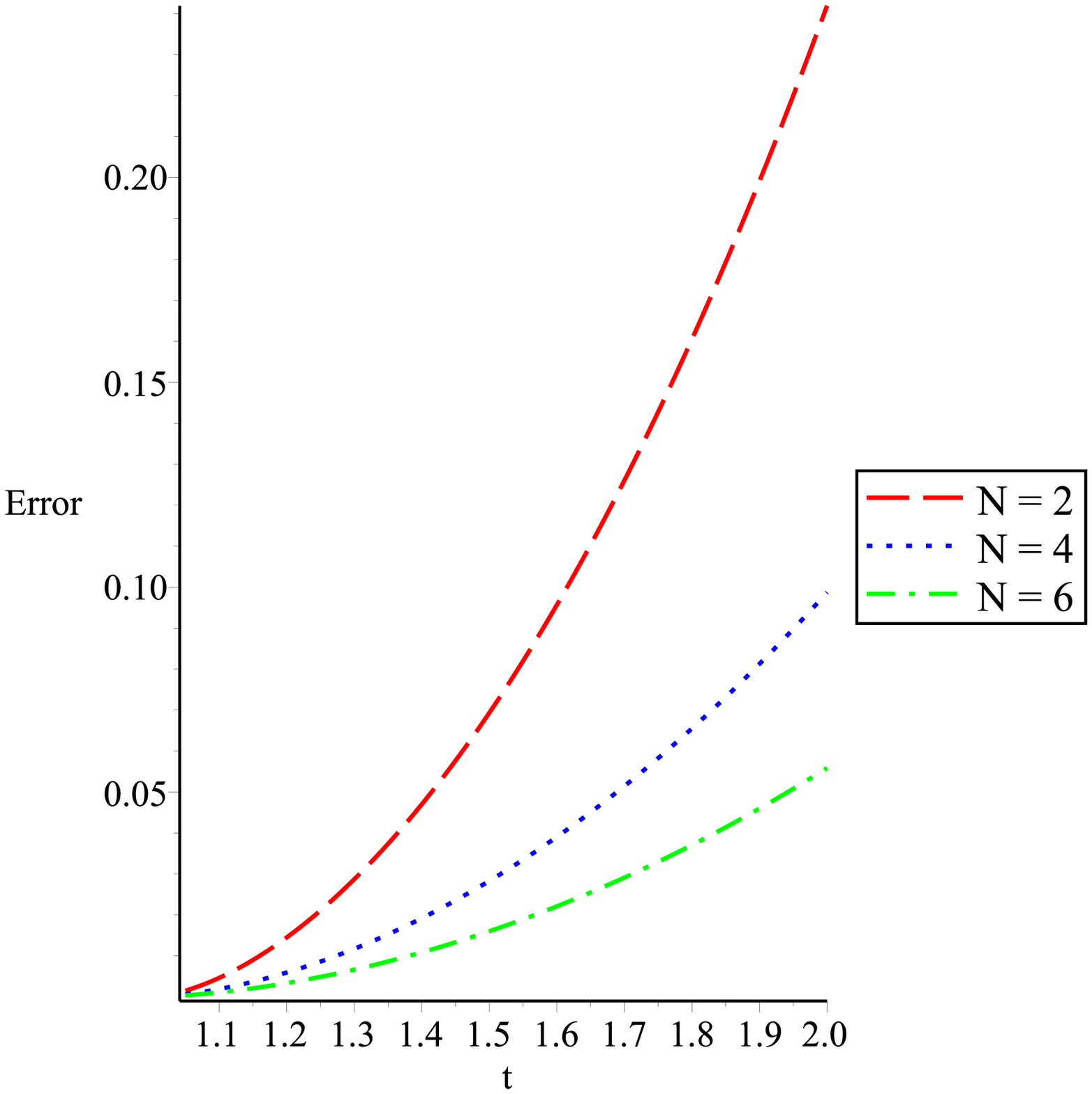}}
\caption{For $\a=0.4$ and $\r=1.5$.}\label{fig:3}
\end{center}
\end{figure}
\end{example}

\begin{example} Consider the following Cauchy type problem
$$\left\{
\begin{array}{l}
\DS {^CD_{1+}^{\a,\r}} x(t)=x(t)+\frac{2\r^\a}{\Gamma(3-\a)}(t^\r-1)^{2-\a}-(t^\r-1)^2, \quad t\in[1,2],\\
x(1)=0,\\
\end{array} \right.$$
It is easy to check that the solution to this system is $x(t)=(t^\r-1)^2$. We shall compare this exact solution with an approximation solutions that can be obtained using the results presented in Section~\ref{sec:approx}. Under the assumption that the set of admissible functions is $C^2([1,2])$ we can replace the fractional operator by the formula given in Theorem~\ref{teo:approx}:
$${^CD_{1+}^{\a,\r}} x(t)\approx A_N(t^\r-1)^{1-\a} t^{1-\r}x'(t)-\sum_{k=1}^N B_{N,k}(t^\r-1)^{1-\a-k}V_k(t),$$
where $A_N$ and $B_{N,k}$ are given as in Theorem~\ref{teo:approx}, and $V_k$ are defined by the solution to the system
$$\left\{
\begin{array}{l}
\DS V'_k(t)=(t^\r-1)^{k-1}x'(t), \quad t\in[1,2],\\
V_k(1)=0,\\
\end{array} \right.$$
for $k=1,\ldots,N$. Thus, we obtain a system of $N+1$ first-order differential equations with $N+1$ initial conditions:
$$\left\{
\begin{array}{l}
\DS  A_N(t^\r-1)^{1-\a} t^{1-\r}x'(t)-x(t)-\sum_{k=1}^N B_{N,k}(t^\r-1)^{1-\a-k}V_k(t) =\frac{2\r^\a}{\Gamma(3-\a)}(t^\r-1)^{2-\a}-(t^\r-1)^2,\\
\DS V'_k(t)=(t^\r-1)^{k-1}x'(t), \quad k=1,\ldots,N,\\
x(1)=0,\\
V_k(1)=0, \quad  k=1,\ldots,N.\\
\end{array} \right.$$
Solutions, for different values of $\a,\r$ and $N$, are presented in Figures \ref{fig:4}, \ref{fig:5} and \ref{fig:6} (on the left side). The error between the exact and the numerical approximations is shown on the right side of figures.

\begin{figure}[h!]
\begin{center}
\subfigure{\includegraphics[scale=0.3]{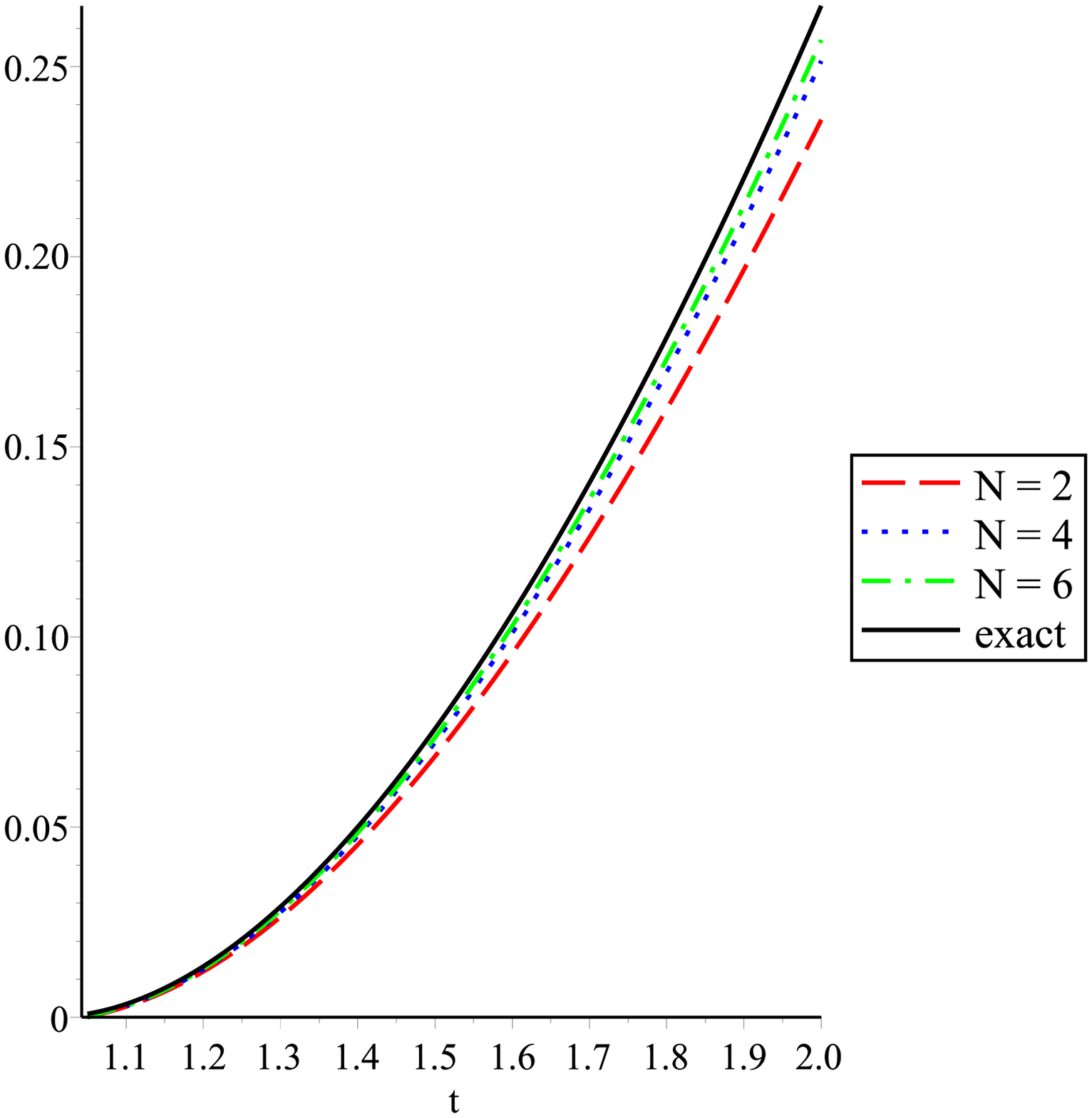}}
\subfigure{\includegraphics[scale=0.3]{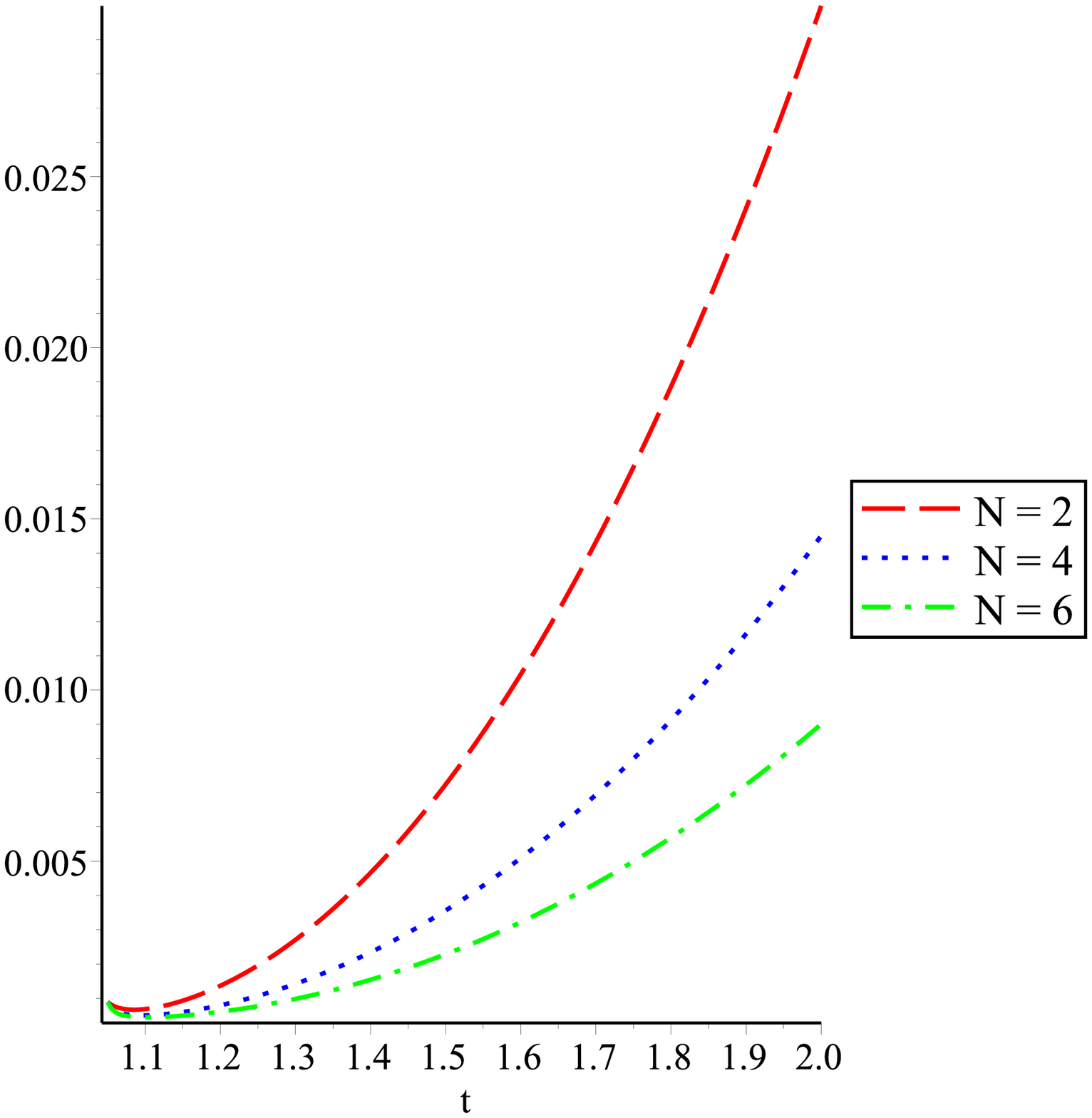}}
\caption{For $\a=0.5$ and $\r=0.6$.}\label{fig:4}
\end{center}
\end{figure}

\begin{figure}[h!]
\begin{center}
\subfigure{\includegraphics[scale=0.3]{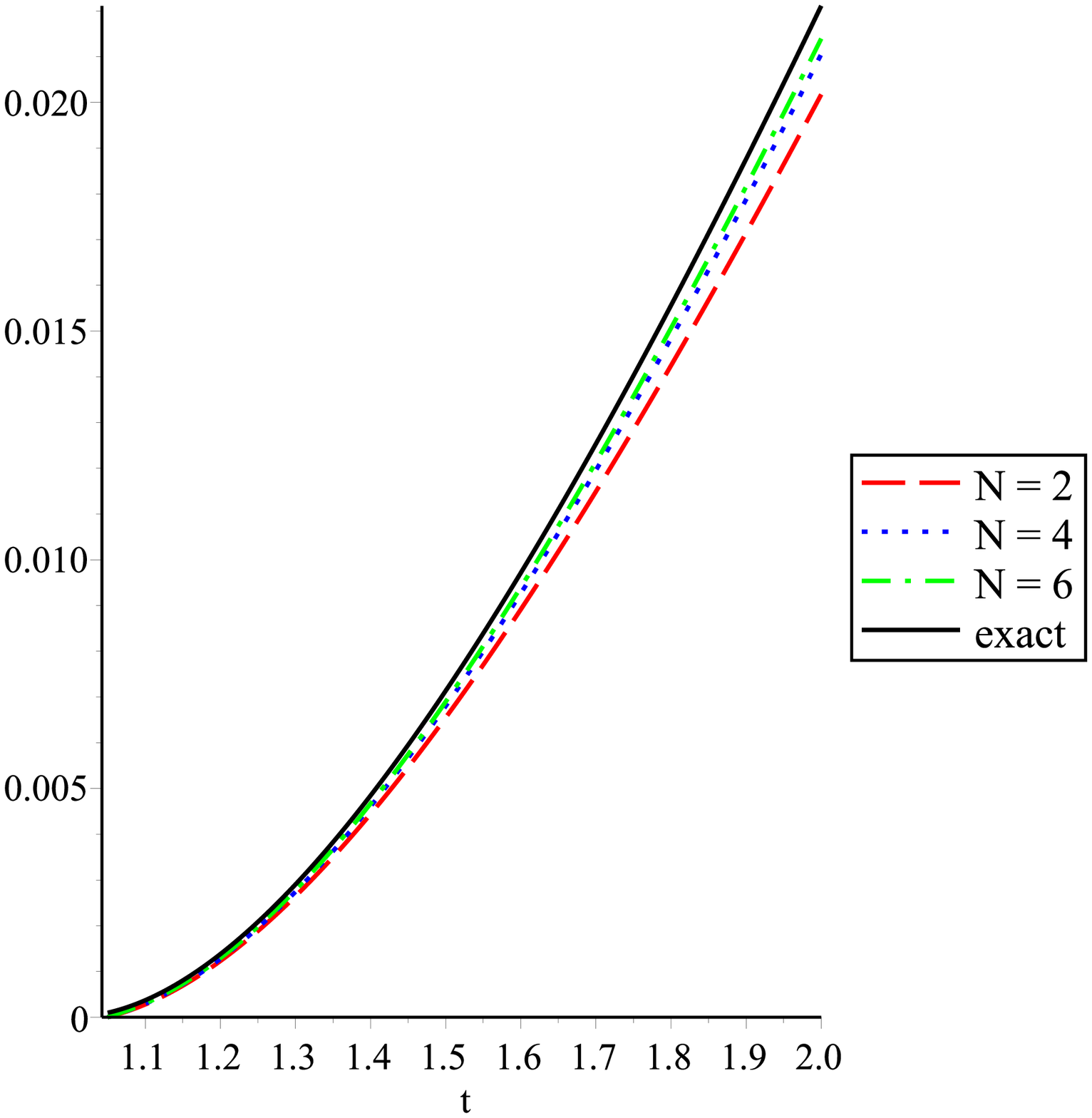}}
\subfigure{\includegraphics[scale=0.3]{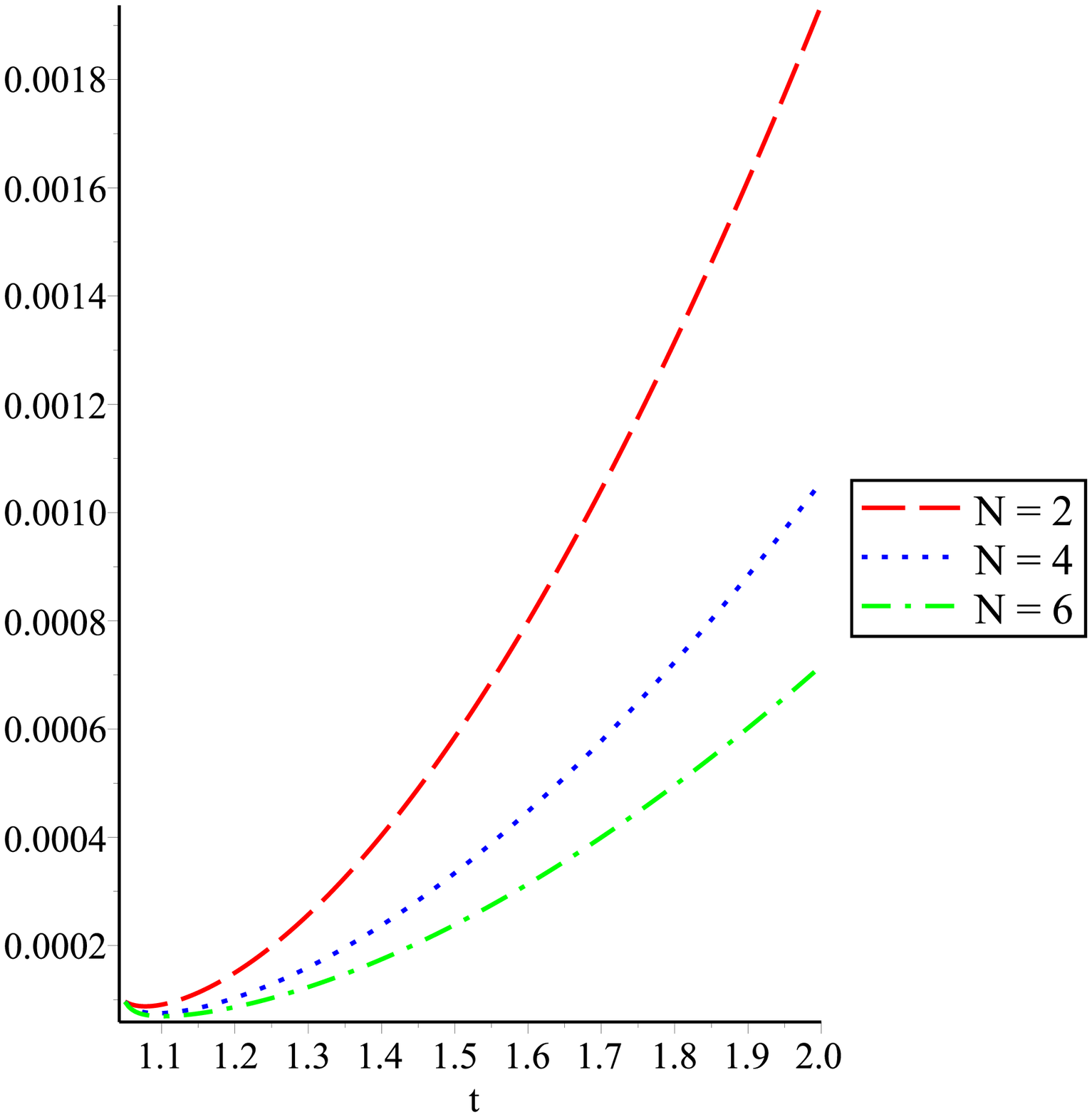}}
\caption{For $\a=0.7$ and $\r=0.2$.}\label{fig:5}
\end{center}
\end{figure}

\begin{figure}[h!]
\begin{center}
\subfigure{\includegraphics[scale=0.3]{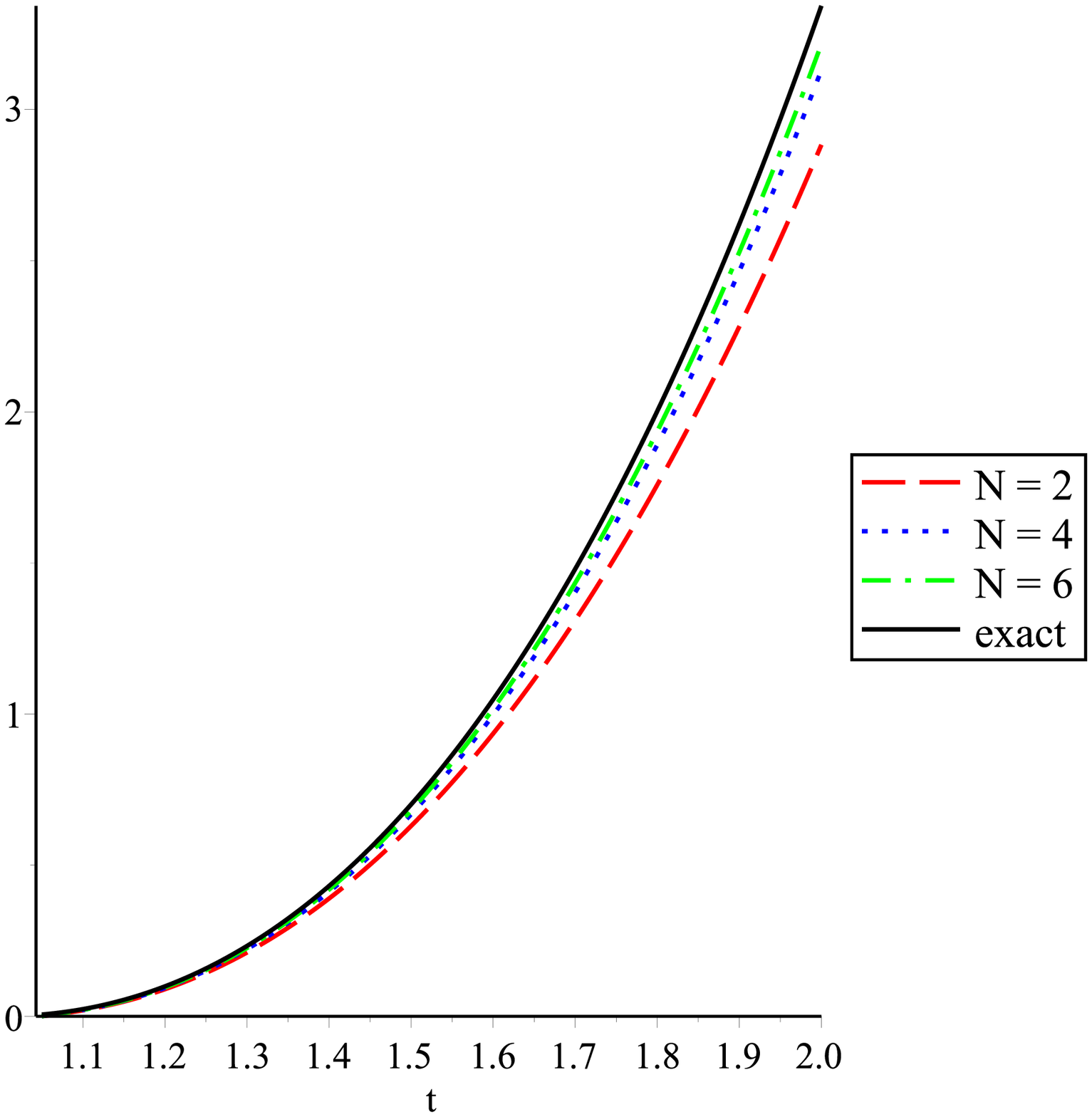}}
\subfigure{\includegraphics[scale=0.3]{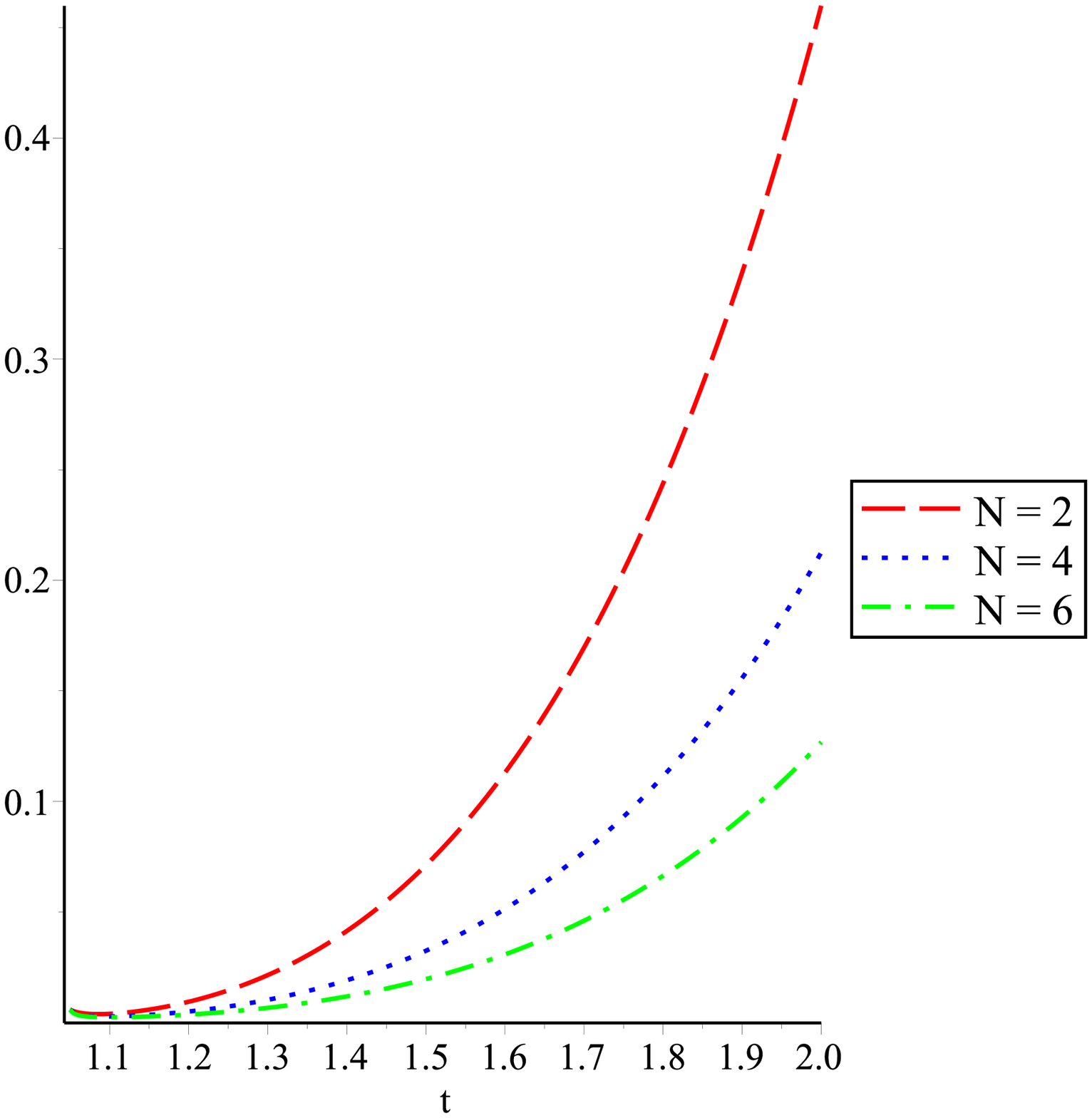}}
\caption{For $\a=0.4$ and $\r=1.5$.}\label{fig:6}
\end{center}
\end{figure}
\end{example}

\begin{example} { The solution to
$$\left\{
\begin{array}{l}
\DS {^CD_{0+}^{\a,\r}} x(t)=\r^\a x(t), \quad t\in[0,1],\\
x(0)=1,\\
\end{array} \right.$$
is the Mittag--Leffler function $x(t)=E_\a(t^{\r\a})$. Replacing the fractional derivative by our approximation formula presented in Theorem~\ref{teo:approx}, we obtain the following system of ordinary differential equations
$$\left\{
\begin{array}{l}
\DS  A_N t^{1-\r \a} x'(t)-\sum_{k=1}^N B_{N,k}t^{\r(1-\a-k)}V_k(t) =\r^\a x(t),\quad t\in[0,1],\\
\DS V'_k(t)=t^{\r(k-1)}x'(t), \quad k=1,\ldots,N,\quad t\in[0,1],\\
x(0)=1,\\
V_k(0)=0, \quad  k=1,\ldots,N.\\
\end{array} \right.$$
Solutions for $N=5,10,15$ and different values of $\a$ and $\r$ are shown in Figures \ref{fig:7} and \ref{fig:8} (on the left side). The errors of the approximation are presented on the right side of the figures.

\begin{figure}[h!]
\begin{center}
\subfigure{\includegraphics[scale=0.3]{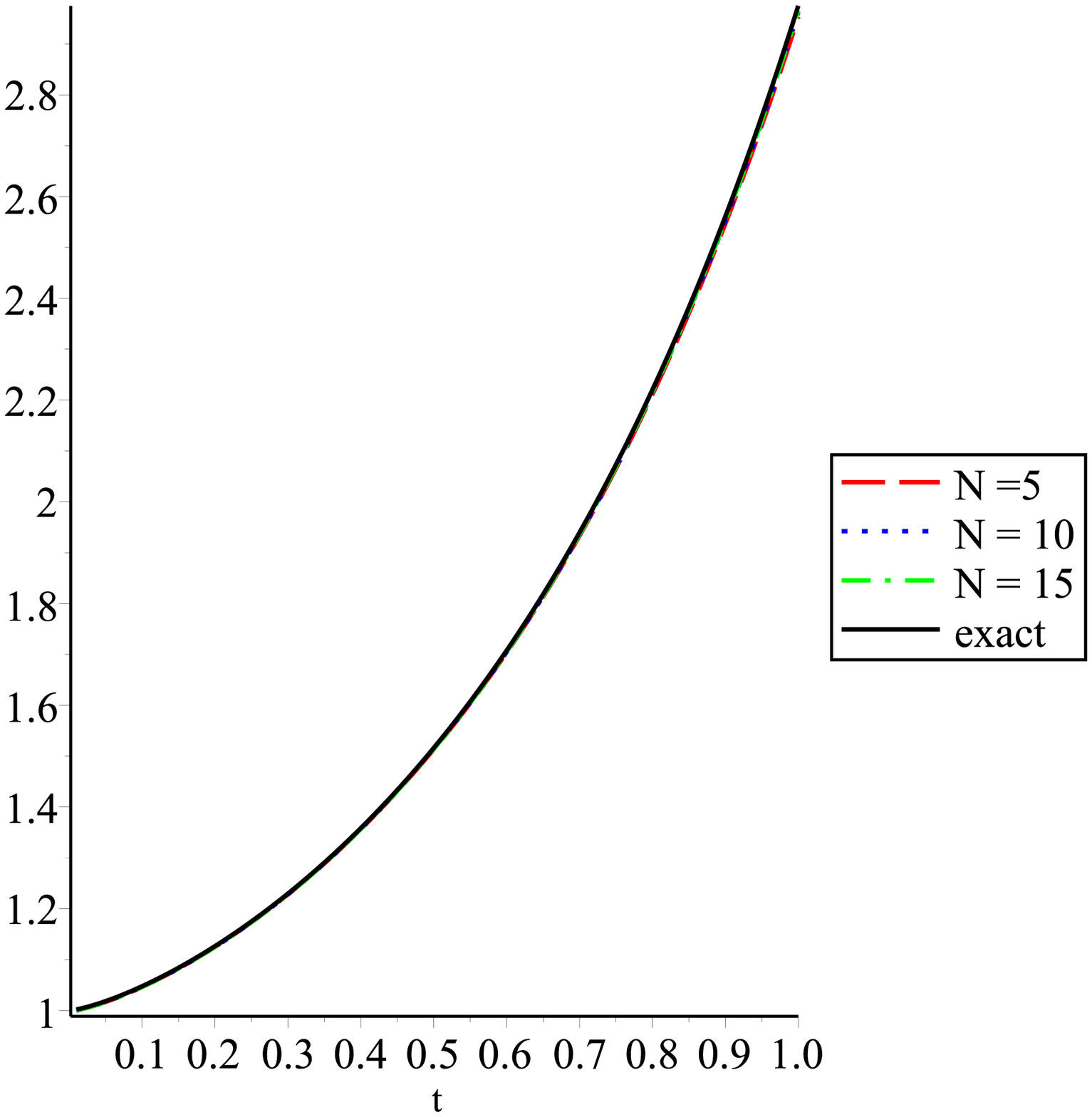}}
\subfigure{\includegraphics[scale=0.3]{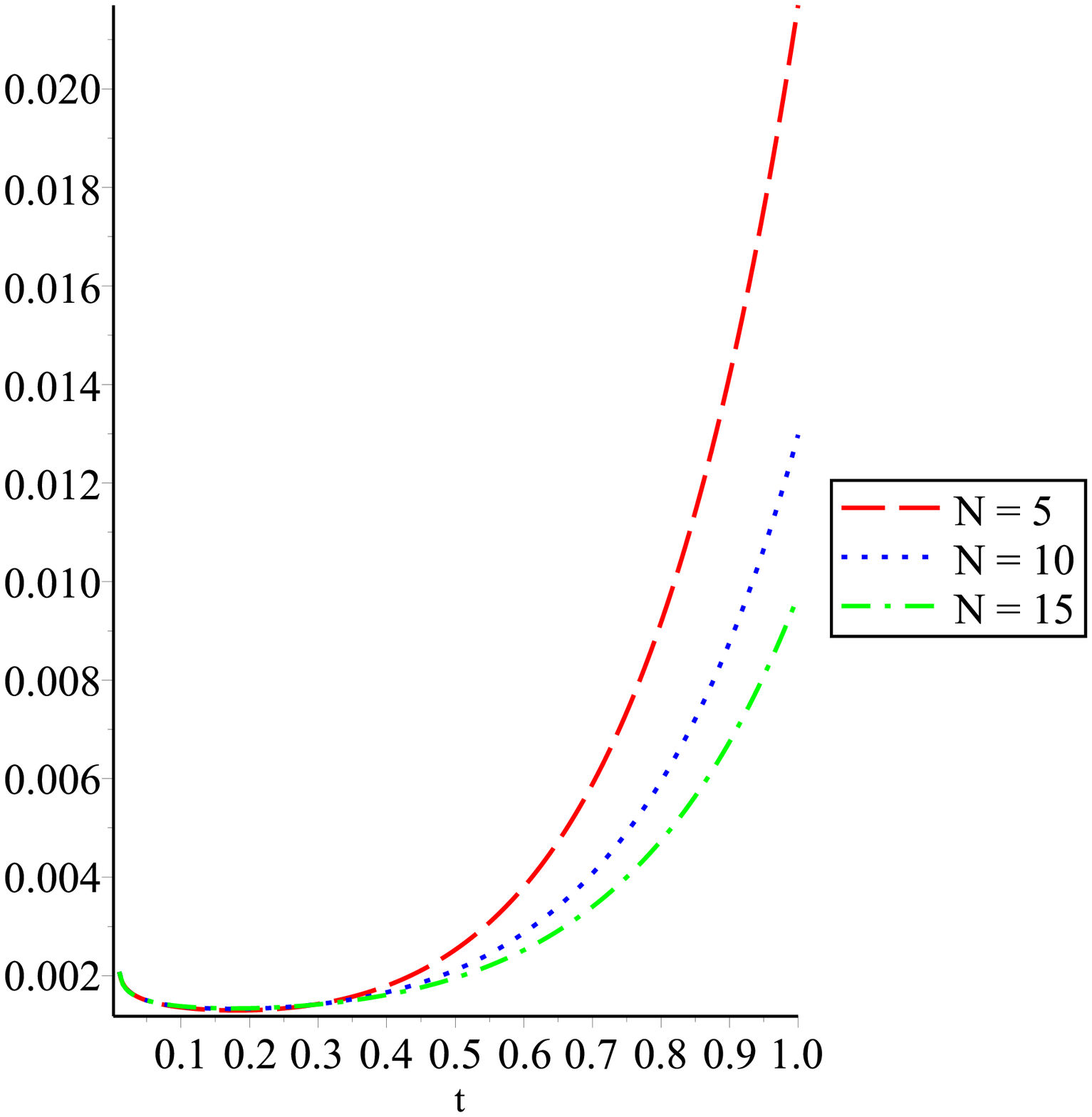}}
\caption{For $\a=0.9$ and $\r=1.5$.}\label{fig:7}
\end{center}
\end{figure}

\begin{figure}[h!]
\begin{center}
\subfigure{\includegraphics[scale=0.3]{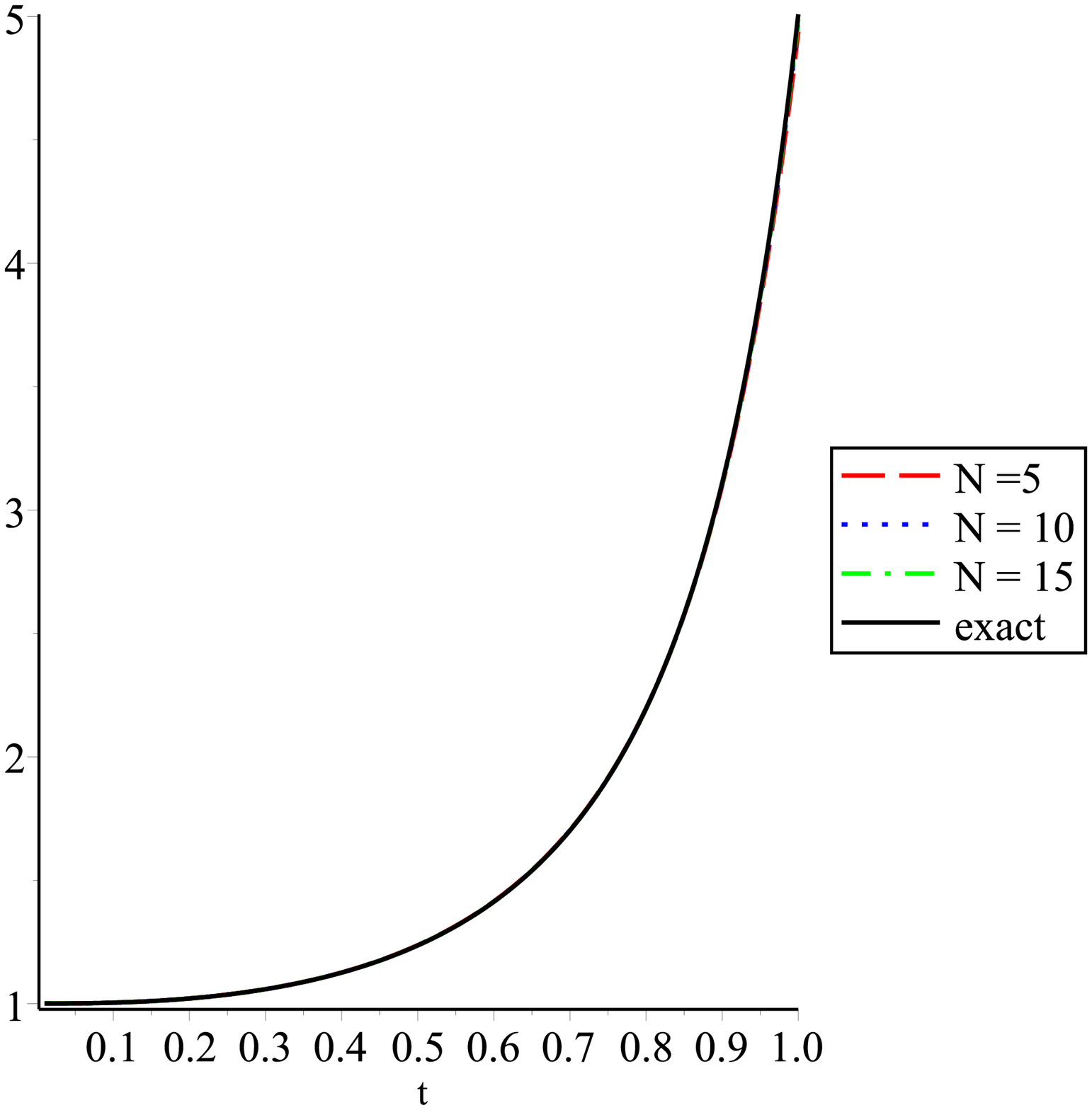}}
\subfigure{\includegraphics[scale=0.3]{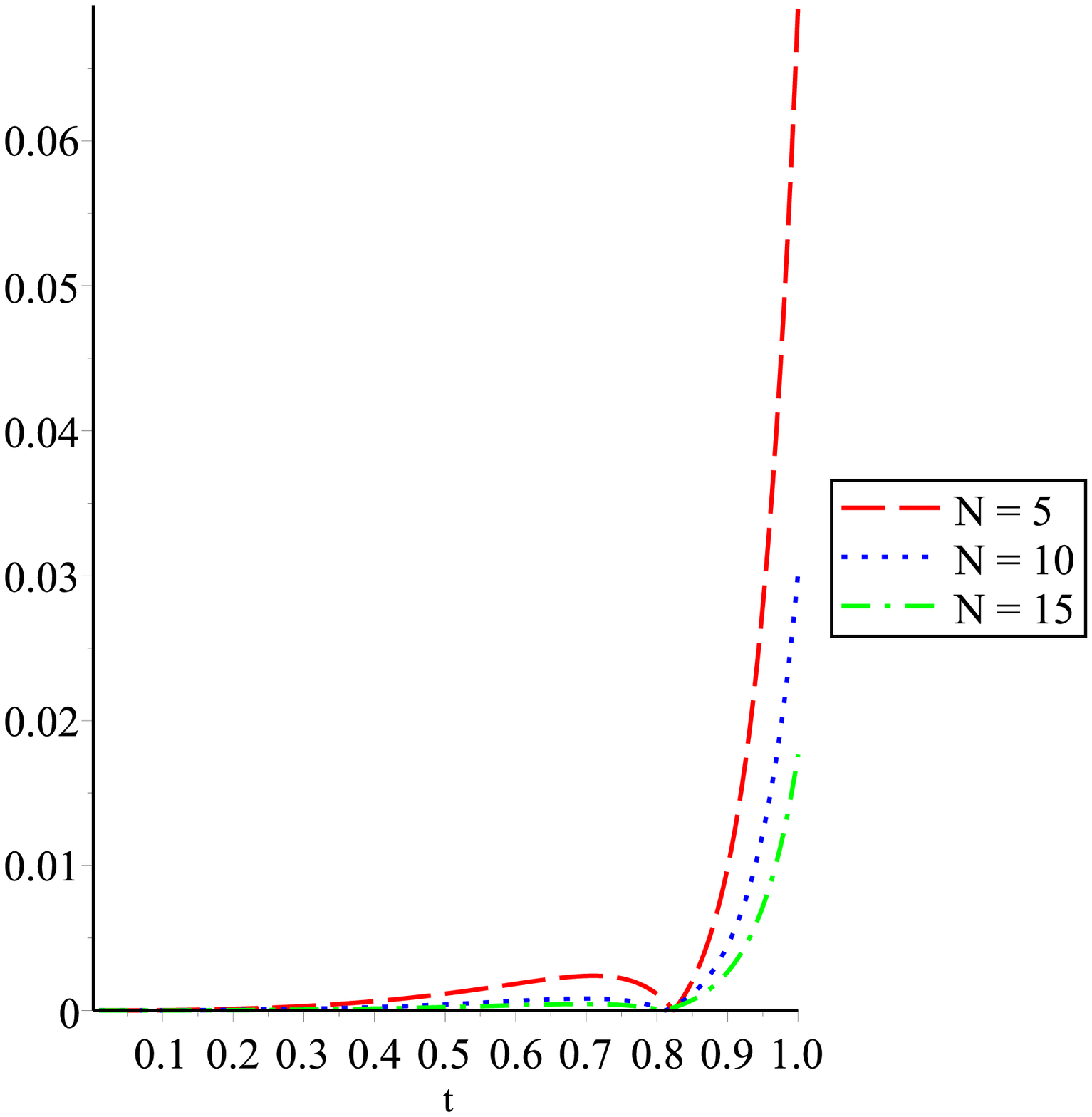}}
\caption{For $\a=0.5$ and $\r=5$.}\label{fig:8}
\end{center}
\end{figure}}
\end{example}

\section{Conclusion}\label{con}

In this paper we present a new fractional operator, which generalizes the Caputo and the Caputo--Hadamard fractional derivative. Some fundamental properties of this operator are presented and proved. Then, an existence and uniqueness theorem for a fractional Cauchy type problem is given.
In order to apply a numerical procedure for solving the mentioned fractional differential problem, we provide a decomposition formula for the new operator in a way that we can rewrite it using only the first-order derivative. Doing so, we can approximate the fractional differential equation by an ordinary one, and thus to determine an approximate solution to the fractional problem. We hope to call the attention of other researchers to consider generalized fractional operators in order to avoid a proliferation of papers devoted to fractional operators.

\section*{Acknowledgments}

The first author is supported by Portuguese funds through the CIDMA - Center for Research and Development in Mathematics and Applications,
and the Portuguese Foundation for Science and Technology (FCT-Funda\c{c}\~ao para a Ci\^encia e a Tecnologia), within project UID/MAT/04106/2013.
A. B. Malinowska is supported by the Bialystok University of Technology grant S/WI/1/2016. T. Odzijewicz by the Warsaw School of Economics grant KAE/S14/35/15.


\end{document}